\documentclass[letterpaper,11pt,twoside]{article}

\usepackage{cmap}
\usepackage[T1]{fontenc}
\usepackage[utf8]{inputenc}
\usepackage[letterpaper,textwidth=5.25in,textheight=8.85in,centering,headheight=13pt,headsep=18pt,footskip=30pt]{geometry}

\usepackage{lmodern}
\usepackage{libertine}

\usepackage{mathrsfs}
\usepackage{amssymb}
\usepackage{amsmath}
\usepackage{amsthm}
\usepackage{microtype}

\usepackage[usenames,dvipsnames]{xcolor}
\definecolor{myblue}{rgb}{0.21, 0.34, 0.74}
\definecolor{mygrey}{rgb}{0.55, 0.57, 0.67}
\definecolor{myred}{rgb}{0.79, 0.0, 0.09}

\usepackage{graphicx}
\usepackage{tikz}
\usetikzlibrary{arrows.meta}
\usepackage{enumitem}
\usepackage{upgreek}
\usepackage{mathtools}
\usepackage{bm}
\usepackage[scr=boondoxo]{mathalfa}

\usepackage{comment}
\usepackage{float}
\usepackage{placeins}
\usepackage{subcaption}
\usepackage[font=small,labelfont=bf,labelsep=period]{caption}

\numberwithin{figure}{section}
\numberwithin{table}{section}

\usepackage[nottoc,notlot,notlof]{tocbibind}
\usepackage{cite}
\usepackage[colorlinks=true,linkcolor=blue,citecolor=blue,urlcolor=blue,bookmarksopen=true]{hyperref}
\usepackage[noabbrev,capitalize,nameinlink]{cleveref}
\usepackage{stmaryrd}

\usepackage{fancyhdr}

\newcommand{\CPAMshorttitle}{A LLOYD-STABILIZED VORONO\"I PARTICLE METHOD}
\newcommand{\CPAMshortauthors}{B. DESPR\'ES AND B. GESHKOVSKI}

\fancypagestyle{cpamfirst}{%
  \fancyhf{}%
  \fancyfoot[L]{}
}

\usepackage{titlesec}
\titleformat{\section}
  {\centering\normalfont\large\bfseries}{\thesection}{0.75em}{}
\titlespacing*{\section}{0pt}{2.5ex plus 0.8ex minus 0.2ex}{1.3ex plus 0.2ex}

\titleformat{\subsection}
  {\normalfont\bfseries}{\thesubsection}{0.75em}{}
\titlespacing*{\subsection}{0pt}{2.0ex plus 0.6ex minus 0.2ex}{0.8ex plus 0.1ex}

\titleformat{\subsubsection}
  {\normalfont\bfseries}{\thesubsubsection}{0.75em}{}
\titleformat{name=\subsubsection,numberless}
  {\normalfont\bfseries}{}{0pt}{}
\titlespacing*{\subsubsection}{0pt}{1.7ex plus 0.5ex minus 0.2ex}{0.6ex plus 0.1ex}

\titleformat{\paragraph}[runin]
  {\normalfont\bfseries}{}{0pt}{}[.]
\titlespacing*{\paragraph}{0pt}{1.2ex plus 0.4ex minus 0.2ex}{0.5em}

\newtheoremstyle{cpamplain}%
  {8pt plus 2pt minus 2pt}
  {8pt plus 2pt minus 2pt}
  {\itshape}
  {}
  {\normalfont\bfseries}
  {.}
  {0.5em}
  {\MakeUppercase{\thmname{#1}}\ \thmnumber{#2}\thmnote{ {\normalfont(#3)}}}%
\newtheoremstyle{cpamdefinition}%
  {8pt plus 2pt minus 2pt}
  {8pt plus 2pt minus 2pt}
  {\normalfont}
  {}
  {\normalfont\bfseries}
  {.}
  {0.5em}
  {\MakeUppercase{\thmname{#1}}\ \thmnumber{#2}\thmnote{ {\normalfont(#3)}}}%

\theoremstyle{cpamplain}
\newtheorem{theorem}{Theorem}[section]
\newtheorem{proposition}[theorem]{Proposition}

\newtheorem{corollary}[theorem]{Corollary}

\theoremstyle{cpamdefinition}

\newtheorem{remark}[theorem]{Remark}

\crefname{theorem}{Theorem}{Theorems}
\crefname{proposition}{Proposition}{Propositions}
\crefname{lemma}{Lemma}{Lemmas}
\crefname{corollary}{Corollary}{Corollaries}
\crefname{definition}{Definition}{Definitions}
\crefname{remark}{Remark}{Remarks}
\crefname{example}{Example}{Examples}

\makeatletter
\renewcommand{\maketitle}{%
  \thispagestyle{cpamfirst}%
  \begingroup
  \centering
  \vspace*{0.56in}%
  {\Large\bfseries\@title\par}%
  \vspace{0.35in}%
  {\normalsize\@author\par}%
  \vspace{0.28in}%
  \endgroup
}
\makeatother

\newcommand{\CPAMauthor}[1]{{\normalsize #1}}
\newcommand{\CPAMaffiliation}[1]{{\itshape\small #1}}
\newcommand{\CPAMand}{{\normalsize AND}}

\renewenvironment{abstract}{%
  \begin{center}\bfseries Abstract\end{center}%
  \small
  \begin{list}{}{\leftmargin=0.55in\rightmargin=0.55in}%
  \item\relax
}{%
  \end{list}\normalsize
}

\newcommand{\R}{\mathbb{R}}

\newcommand{\diff}{\mathrm{d}}
\newcommand{\dist}{\mathrm{dist}}
\newcommand{\Secref}[1]{\textbf{\hyperref[#1]{Section~\ref*{#1}}}}
\newcommand{\Appref}[1]{\textbf{\hyperref[#1]{Appendix~\ref*{#1}}}}

\let\alpha\upalpha
\let\beta\upbeta
\let\gamma\upgamma
\let\delta\updelta
\let\epsilon\upepsilon
\let\varepsilon\upvarepsilon
\let\zeta\upzeta
\let\eta\upeta
\let\theta\uptheta
\let\vartheta\upvartheta
\let\iota\upiota
\let\kappa\upkappa
\let\lambda\uplambda
\let\mu\upmu
\let\nu\upnu
\let\xi\upxi
\let\pi\uppi
\let\varpi\upvarpi
\let\rho\uprho
\let\varrho\upvarrho
\let\sigma\upsigma
\let\varsigma\upvarsigma
\let\tau\uptau
\let\upsilon\upupsilon
\let\phi\upphi
\let\varphi\upvarphi
\let\chi\upchi
\let\psi\uppsi
\let\omega\upomega
\let\Gamma\Upgamma
\let\Delta\Updelta
\let\Theta\Uptheta
\let\Lambda\Uplambda
\let\Xi\Upxi
\let\Pi\Uppi
\let\Sigma\Upsigma
\let\Upsilon\Upupsilon
\let\Phi\Upphi
\let\Psi\Uppsi
\let\Omega\Upomega

\renewcommand{\leq}{\leqslant}

\renewcommand{\le}{\leqslant}
\renewcommand{\ge}{\geqslant}

\numberwithin{equation}{section}

\hypersetup{pdftitle={A Lloyd-stabilized Voronoi particle method},pdfauthor={Bruno Despres and Borjan Geshkovski}}

\title{A Lloyd-stabilized Vorono\"i particle method}
\author{%
\CPAMauthor{BRUNO DESPR\'ES}\\[-1pt]
\CPAMaffiliation{Laboratoire Jacques-Louis Lions}\\[0.7em]
\CPAMand\\[0.7em]
\CPAMauthor{BORJAN GESHKOVSKI}\\[-1pt]
\CPAMaffiliation{Laboratoire Jacques-Louis Lions}%
}
\date{}

\begin{document}
\setlist[itemize,enumerate]{leftmargin=2em,itemsep=0pt,parsep=0pt}

%
%

\maketitle

%
%

\begin{abstract}
We study the stability and consistency of Lloyd's algorithm used in combination with the Lagrangian transport of a density on a Vorono\"i tessellation.
We show that relaxation rates as strong as $O(h^{-1/2})$, where $h$ is the mesh size, give convergence to the continuity equation in Wasserstein distance at the rate $O(h^{1/4})$. Thus the mesh can be kept regular by the correction alone, without the remeshing that Lagrangian methods usually require. An application is proposed for the compressible Euler equations.
\end{abstract}

\thispagestyle{cpamfirst}

\section{Introduction}

\emph{Vorono\"i methods} discretize a fluid by particles, and use the Vorono\"i tessellation generated by the particles as a moving mesh for a finite volume scheme. They have a long history \cite{pasta,RUSSO199384,PhysRevLett.83.1775,DEVILLERS1996315}, and have seen a surge of activity since the work of Springel \cite{Springel2010}; see \cite{shasha,loulou,DUQUE2023268,DECAMPOS2022114680,GABURRO2020109167} and, for incompressible flows, \cite{kincl}. In this work we consider the compressible Lagrangian Vorono\"i method of \cite{despres2024,desprem}. Like every Lagrangian method, it suffers from mesh degeneration, since near shocks, particles come very close to each other and the Vorono\"i cells are collapsed. The usual remedy is of ALE type, in which the mesh is rezoned from time to time and the solution is remapped onto the new mesh \cite{loulou,kincl}. Here we propose a different stabilization, in which after each time step, every particle is moved towards the centroid of its cell, keeping the method purely Lagrangian. This is one step of Lloyd's algorithm \cite{Lloyd1982}, the classical iteration for computing centroidal Vorono\"i tessellations, and also the update rule of $k$-means clustering \cite{MacQueen1967}.

Our goal in this work is to understand how much this correction can stabilize the mesh, and to find the most general conditions which preserve the consistency of the method. 
Our tools involve three geometric indicators of a Vorono\"i tessellation, and inequalities between them. One indicator is the quantization energy of the particles, which is also the semi-discrete optimal transport cost between the uniform measure on $\Omega$ and the particles. Semi-discrete optimal transport on Vorono\"i and Laguerre cells is at the heart of the Lagrangian schemes of M\'erigot and Mirebeau \cite{MerigotMirebeau2016} and of Gallou\"et and M\'erigot \cite{GallouetMerigot2018,gaga,caca}, where the transport cost acts as a pressure rather than as a stabilizer. 

The practical use of Lloyd's algorithm in a numerical method for fluids with strong gradients, or even with shocks, would call for velocity fields of weak regularity such as $v\in \mathrm{BV}\left((0,T)\times \Omega;\R^2\right)$, and this is indeed the regularity we meet in our numerical experiments on the compressible Euler equations.
The mathematical theory is far more demanding at that level of regularity, and we therefore work throughout with $v\in L^1((0,T);C^{0,1}(\Omega;\R^2))$.
This choice is deliberate as it reduces every transport estimate to an elementary Gronwall argument, so that the mechanism of the stabilization, and the interplay between the three geometric indicators introduced below, stay visible rather than buried in technicalities.

\subsection{Moving Vorono\"i tessellations}\label{sec:moving}

Throughout the paper $\Omega$ is either a bounded convex Lipschitz domain of $\mathbb R^2$ or a flat torus $\Omega=\R^2/\mathcal L$, where $\mathcal L:=\mathbb Z\ell_1+\mathbb Z\ell_2$ and $\ell_1,\ell_2\in\R^2$ are linearly independent.
On the torus, functions on $\Omega$ are identified with $\mathcal L$-periodic functions on $\R^2$, a point of $\Omega$ with any of its representatives in $\R^2$, and $|x-y|$ denotes the distance on the torus, so that if $\tilde x,\tilde y\in\R^2$ are representatives of $x,y\in\Omega$, then
\(|x-y|=\min_{k\in\mathcal L}|\tilde x-\tilde y-k|,\)
where on the right-hand side $|\cdot|$ is the Euclidean norm of $\R^2$. Balls and diameters are understood accordingly.
At a discrete time $t^k$
we take a collection of $N$ distinct particles (or generators)
$
X^k=(x_1^k,\dots,x_N^k)\in\Omega^N.
$
The particles are distinct, that is $x_i^k\neq x_j^k$ for $i\neq j$.
We denote by $\{\Omega_i^k\}_{i=1}^N$ the associated Vorono\"i tessellation of $\Omega$
\[
\Omega_i^k:=\{x\in\Omega:\ |x-x_i^k|< |x-x_j^k| \mbox{ for }1\leq  j\leq N,\ j\neq i\}, \quad 1\leq i \leq N.
\]
Each cell carries a fixed conserved mass $M_i> 0$ (initialized at $k=0$ and then kept constant). We define
a local density $\rho_i^k:=\frac{M_i}{|\Omega_i^k|} >0$. It defines a global density field
\begin{equation}\label{eq:rho-def}
\rho^k(x):=\sum_{i=1}^N \rho_i^k\,\mathbf 1_{\Omega_i^k}(x).
\end{equation}
Then, given velocities $v_i^k$ for $1\leq i \leq N$, the particles evolve to the next discrete time $t^{k+1}$ via
\begin{equation}\label{eq:lag-step-0}
x_i^{k+1,-}=x_i^k+\Delta t^k\,v_i^k,
\qquad 1\le i\le N.
\end{equation}
In the theoretical part of this work, the velocities are those of a given velocity field 
$$
v_i^k=v(t^k, x_i^k), \qquad v:[0,T)\times \Omega \to \R^2.
$$
In numerical applications they are instead computed, for instance by a Riemann solver for the compressible Euler equations, and the time step $\Delta t^k>0$ satisfies a CFL-type restriction.
In both cases, the positions $X^{k+1,-}$ at time $t^{k+1}$ determine a new Vorono\"i tessellation $\{\Omega_i^{k+1,-}\}$.
The density field is obtained by Lagrangian transport of the generators, that is
\begin{equation}\label{eq:rho-def-1}
\rho^{k+1,-}(x):=\sum_{i=1}^N \rho_i^{k+1,-}\,\mathbf 1_{\Omega_i^{k+1,-}}(x), \qquad \rho_i^{k+1,-}:=\frac{M_i}{|\Omega_i^{k+1,-}|} >0.
\end{equation}

\subsection{Stabilizing \texorpdfstring{\eqref{eq:lag-step-0}}{(1.2)} with Lloyd's algorithm}\label{sec:lloyd}

There are many situations reported in the literature where the displacement of the Vorono\"i tessellation and the Lagrangian transport of
the density lead to certain kinds of numerical instabilities, see \cite{shasha,Springel2010,GABURRO2020109167,kincl,despres2024,GallouetMerigot2018}.
We propose to mitigate these instabilities by correcting the step \eqref{eq:lag-step-0} with one step of Lloyd's algorithm.
Figure~\ref{fig:lloyd-summary} shows the effect of this correction on a long computation.

\begin{figure}[t]
\centering
\begin{tikzpicture}
\node[inner sep=0pt] (A) at (0,0)
  {\includegraphics[width=0.40\linewidth,trim=43.8 38 15.5 27.2,clip]{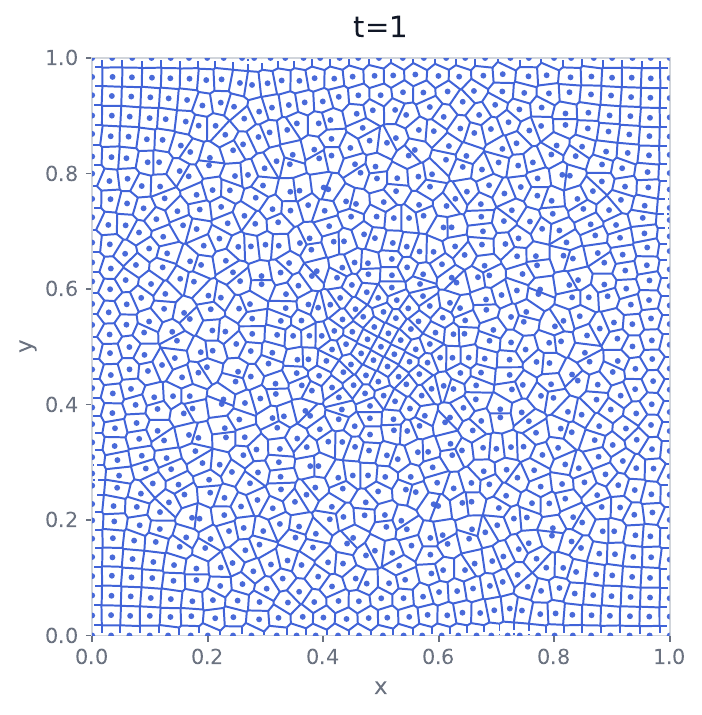}};
\node[inner sep=0pt] (B) at (0.60\linewidth,0)
  {\includegraphics[width=0.40\linewidth,trim=43.8 38 15.5 27.2,clip]{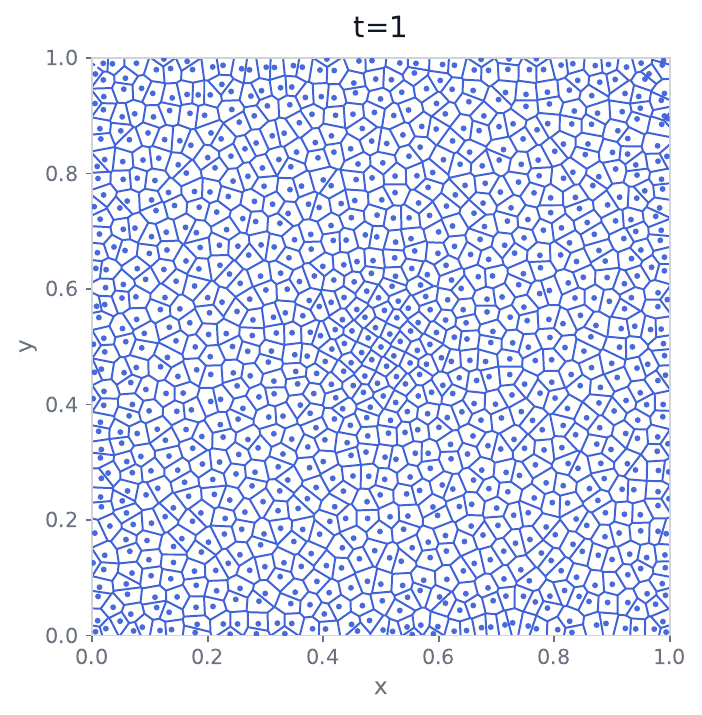}};
\draw[-{Latex[length=2.6mm,width=2mm]},line width=0.9pt]
  ([xshift=3.5mm]A.east) -- ([xshift=-3.5mm]B.west)
  node[midway,above=2pt,font=\small] {Lloyd};
\end{tikzpicture}
\caption{The same computation without and with the Lloyd correction, for a vortex solution of the compressible Euler equations, computed with the method of \cite{despres2024,desprem} on an initially Cartesian mesh, both at time $t=1$.
Without correction (left), particles come together and many cells collapse.
With correction (right), the mesh stays close to a centroidal Vorono\"i tessellation.}
\label{fig:lloyd-summary}
\end{figure}

Let $c_i^{k+1,-}$ be the centroid of $\Omega_i^{k+1,-}$
\[
c_i^{k+1,-}:=\frac{1}{|\Omega_i^{k+1,-}|}\int_{\Omega_i^{k+1,-}} x\,\diff x.
\]
The Lloyd step moves each particle towards the centroid of its cell
\begin{equation}\label{eq:lloyd-step}
x_i^{k+1}=x_i^{k+1,-}+\eta_i^k\left(c_i^{k+1,-}-x_i^{k+1,-}\right),
\qquad 0\le \eta_i^k\le 1.
\end{equation}
The final Vorono\"i tessellation generated by $X^{k+1}$ is $\{\Omega_i^{k+1}\}$.
The final density field is
\[
\rho^{k+1}(x):=\sum_{i=1}^N \rho_i^{k+1}\,\mathbf 1_{\Omega_i^{k+1}}(x), \qquad \rho_i^{k+1}:=\frac{M_i}{|\Omega_i^{k+1}|}>0.
\]

\subsection{Continuous-time viewpoint}\label{sec:continuous}

If the parameters in \eqref{eq:lloyd-step} are scaled as $\eta_i^k=\alpha(t^k)\,\Delta t^k$ with $\alpha\ge 0$
(and $\Delta t^k\to0$), then the split update \eqref{eq:lag-step-0}--\eqref{eq:lloyd-step}
formally corresponds to 
\begin{equation}\label{eq:forced-lloyd-ode}
\dot x_i^h(t)=v_i^h(t)+\alpha(t)\left(c_i(X^h(t))-x_i^h(t)\right),
\qquad 1\le i\le N,
\end{equation}
where $v_i^h$ denotes the discrete velocity produced by the physical finite-volume step and $c_i(X^h)$ denotes the centroid of the Vorono\"i cell generated by $X^h$.
In the idealized case $v_i^h(t)=v(t,x_i^h(t))$, \eqref{eq:forced-lloyd-ode} reduces to the exact-advection model studied in \Secref{sec:quasi-uniformity}.
If $\alpha=0$, one expects, and one indeed observes in computations, that shocks in the velocity field lead to strongly distorted and eventually degenerate Vorono\"i cells. We therefore regularize the physical advection by the Lloyd stabilizer in \eqref{eq:forced-lloyd-ode}.

The smoothing effect of the correction is particularly transparent in one dimension.
Consider $N$ particles $0\le x_1<\dots<x_N<1$ on $\R/\mathbb Z$ with $x_{i+N}:=x_i+1$ for all $i$.
The Vorono\"i cell of $x_i$ is the interval between
$(x_{i-1}+x_i)/2$ and $(x_i+x_{i+1})/2$, so its centroid satisfies
\[
c_i-x_i=\frac14\bigl(x_{i+1}-2x_i+x_{i-1}\bigr).
\]
Thus the Lloyd correction is proportional to a second-order finite difference of the positions, and acts as a discrete diffusion.
Writing $h=1/N$, it is of order $\alpha h^2$ when the positions follow a smooth function of the particle label, and its damping rate on oscillations is parabolic. Indeed, for the Lloyd dynamics alone, a perturbation $x_i=ih+b(t)\cos(\theta i)$ has amplitude satisfying
\(\dot b(t)=-\alpha(t)\sin^2(\theta/2)b(t).
\)
A mode of fixed physical wavelength has $\theta=O(h)$ and decays at rate
$O(\alpha h^2)$, whereas a mode varying on the scale of the spacing decays
at rate $O(\alpha)$.
The correction should therefore be strong enough to damp oscillations between neighboring particles, but weak enough to preserve smooth physical motion. With $\alpha_h=h^{-1}$, for instance, the former decay at rate $O(h^{-1})$ and the latter at rate $O(h)$.

Proving both mesh control and consistency in two dimensions is considerably more involved, and requires further estimates.
The rest of the paper analyzes \eqref{eq:forced-lloyd-ode} and its consequences.

\subsection{Our contributions}\label{sec:contributions}

Our analysis is based on inequalities between three natural indicators of the quality of a Vorono\"i tessellation.
For a configuration $X\in\Omega^N$ with pairwise distinct generators, let $\{\Omega_i(X)\}_{1\le i\le N}$ be the corresponding Vorono\"i tessellation, with
volumes $V_i(X):=|\Omega_i(X)|$ and centroids $c_i(X):=V_i(X)^{-1}\int_{\Omega_i(X)}x\,\diff x$.
The first indicator is the energy
\begin{equation}\label{eq:energy-F}
\mathscr F(X):=\sum_{i=1}^N\int_{\Omega_i(X)}|x-x_i|^2\,\diff x,
\end{equation}
the second indicator is the modified energy
\begin{equation}\label{eq:energy-G}
\mathscr G(X):=\sum_{i=1}^N V_i(X)|x_i-c_i(X)|^2,
\end{equation}
and the third indicator is the characteristic length of the tessellation
\begin{equation}\label{eq:energy-D}
\mathscr D(X):=  \max_{1\leq i \leq N} \operatorname{diam}(\Omega_i(X)).
\end{equation}
Here and below, $h:=\sqrt{|\Omega|/N}$ is the mesh size, so that $h^2$ is the mean cell volume.

We say that the Vorono\"i tessellation remains \emph{quasi-uniform} on $[0,T]$ if $\mathscr D(X(t))\le Ch$ for $t\in[0,T]$, with $C$ independent of $h$.
This only says that no cell is much larger than the mean cell; it is weaker than quasi-uniformity in the finite element sense, which also requires that no cell is much smaller than the mean cell (see Corollary~\ref{cor:mesh-bounds} for the latter).
Several of the difficulties we meet come down to obtaining good bounds on $t\mapsto \mathscr D(X(t))$.

Our results are the following.

\begin{itemize}
\item
In \Secref{sec:convergence} we prove that the approximate density, namely the Vorono\"i histogram generated by \eqref{eq:forced-lloyd-ode}, converges to the solution of the continuity equation \eqref{eq:continuity}, provided the discrete velocities converge to $v$, the cell diameters tend to zero, and the accumulated Lloyd displacement defined in \eqref{eq:lloyd_defect_weighted} tends to zero (Theorem~\ref{thm:W1-convergence}).
\smallskip

\item
In \Secref{sec:quasi-uniformity} we study \eqref{eq:forced-lloyd-ode} with the exact velocity $v\in L^1((0,T);C^{0,1}(\Omega;\R^2))$.
We first show that a feedback choice of $\alpha$ traps the configuration near the minimizers of $\mathscr F$, which yields quasi-uniformity (Theorem~\ref{thm:alpha-feedback}) and quantitative bounds on $B_h(T)$ (Corollary~\ref{cor:mesh-bounds}).
We find it difficult to foresee a numerical scenario which would take advantage of this, and we therefore regard the quasi-uniformity produced by Lloyd's algorithm as a structural rather than a practical feature.
\smallskip

\item In \Secref{subsec:adaptive-energy} we show that any locally Lipschitz feedback with
$0\leq \alpha_h\leq \mathscr G(X^h)/h^{5/2}=O(h^{-1/2})$
provides consistency in Wasserstein distance (Theorem~\ref{thm:adaptive-energy-convergence}), with a rate not worse than $O(h^{1/4})$; even a strong relaxation rate therefore preserves consistency.
This seems to us an important consequence of our work, since it opens the way to using Lloyd's algorithm in numerical fluid solvers without the remeshing usually performed, as for instance in \cite{shasha,loulou}.
\smallskip

\item
In \Secref{sec:osl} we discuss why most of the preceding conclusions hold for one-sided Lipschitz velocities as well.
\smallskip

\item
In \Secref{sec:num} we illustrate the effect of the Lloyd stabilization on a numerical solution of the compressible Euler equations, computed with the method of \cite{despres2024,desprem}. 
\end{itemize}

\section{Convergence of the Vorono\"i histogram}\label{sec:convergence}

Up to \Secref{sec:osl}, the velocity field is $v\in L^1((0,T);C^{0,1}(\Omega;\R^2))$,\footnote{Only the one-sided part of the Lipschitz bound is really used, as explained in \Secref{sec:osl}.} with $v\cdot n=0$ on $\partial\Omega$ when $\Omega$ is a bounded domain. We write
\(
L(t):=\|v(t,\cdot)\|_{C^{0,1}(\Omega)},
\Lambda(t):=\int_0^t L(s)\,\diff s.
\)
We consider the continuity equation
\begin{equation}\label{eq:continuity}
\partial_t\rho+\nabla\cdot(\rho v)=0\qquad\text{in }(0,T)\times\Omega,
\end{equation}
with initial datum $\rho_0\in L^1(\Omega)$, $\rho_0\ge0$.
Since $v$ is Lipschitz in $x$, \eqref{eq:continuity} has a unique weak solution $\rho$, given by $\mu(t):=\rho(t,\cdot)\,\diff x=\Phi(t,\cdot)_\#\mu(0)$, where $\Phi$ is the flow of $v$.
Our goal is to compare $\mu$ with the Vorono\"i histogram generated by \eqref{eq:forced-lloyd-ode}, where $v_i^h$ are the discrete velocities of the scheme.
The comparison is made in the Wasserstein distance \cite{villani}
\[
W_1(\mu,\nu):=\inf_{\pi\in\Pi(\mu,\nu)}\int_{\Omega\times\Omega}|x-y|\,\diff\pi(x,y),
\]
where $\mu$ and $\nu$ are nonnegative measures on $\Omega$ with the same mass and $\Pi(\mu,\nu)$ is the set of couplings of $\mu$ and $\nu$, that is, of nonnegative measures on $\Omega\times\Omega$ with first marginal $\mu$ and second marginal $\nu$.

\begin{theorem}\label{thm:W1-convergence}
Let $X^h(t)=(x_1^h(t),\dots,x_N^h(t))$ be an absolutely continuous solution of \eqref{eq:forced-lloyd-ode} on $[0,T]$, where $\alpha\in L^1(0,T)$ is nonnegative and $v_i^h\in L^1((0,T);\R^2)$.
Assume that the generators remain distinct on $[0,T]$, and let $\{\Omega_i^h(t)\}_{1\le i\le N}$ be the associated Vorono\"i cells.
Let $\mu^h(t):=\rho^h(t,\cdot)\,\diff x$ be the Vorono\"i histogram \eqref{eq:rho-def} with masses $M_i^h:=\int_{\Omega_i^h(0)}\rho_0\,\diff x$, that is
\[
\rho^h(t,x):=\sum_{i=1}^N\frac{M_i^h}{|\Omega_i^h(t)|}\mathbf 1_{\Omega_i^h(t)}(x),
\qquad
M:=\sum_{i=1}^N M_i^h=\int_\Omega \rho_0\,\diff x .
\]
Let $d_h:=\sup_{t\in[0,T]}\mathscr D(X^h(t))$ be the largest cell diameter on $[0,T]$, let $\mathscr E_h(t):=\sum_i M_i^h\,|v_i^h(t)-v(t,x_i^h(t))|$ be the velocity residual, and let
\begin{equation}\label{eq:lloyd_defect_weighted}
B_h(t):=\int_0^t e^{\Lambda(t)-\Lambda(s)}\alpha(s)
\sum_{i=1}^N M_i^h\left|c_i(X^h(s))-x_i^h(s)\right|\,\diff s.
\end{equation}
Then, for every $t\in[0,T]$,
\begin{equation}\label{eq:W1_conv_bound}
W_1\left(\mu^h(t),\mu(t)\right)
\le M\left(1+e^{\Lambda(t)}\right)d_h
+\int_0^t e^{\Lambda(t)-\Lambda(s)}\mathscr E_h(s)\,\diff s
+B_h(t).
\end{equation}
In particular, if
$
d_h\to0$,
$\int_0^T\mathscr E_h(s)\,\diff s\to 0 $ and 
$B_h(T)\to0$, 
then 
$$\sup_{t\in[0,T]}W_1(\mu^h(t),\mu(t))\to 0.$$
\end{theorem}

\begin{proof}
We compare the Vorono\"i histogram with the exact solution through the atomic measure
\[
\nu^h(t):=\sum_{i=1}^N M_i^h\,\delta_{x_i^h(t)} .
\]
Fix $t\in[0,T]$. The measure
\[
\pi_t(\diff x,\diff y):=\sum_{i=1}^N\frac{M_i^h}{|\Omega_i^h(t)|}\mathbf 1_{\Omega_i^h(t)}(x)\,\diff x\otimes\delta_{x_i^h(t)}(\diff y)
\]
has first marginal $\rho^h(t,\cdot)\,\diff x=\mu^h(t)$ and second marginal $\nu^h(t)$, so that $\pi_t\in\Pi(\mu^h(t),\nu^h(t))$.
Since $x_i^h(t)\in\Omega_i^h(t)$, we have $|x-x_i^h(t)|\le\mathscr D(X^h(t))\le d_h$ for $x\in\Omega_i^h(t)$, whence
\begin{equation}\label{eq:mu_to_nu}
\begin{aligned}
W_1\left(\mu^h(t),\nu^h(t)\right)
&\le\int_{\Omega\times\Omega}|x-y|\,\diff\pi_t(x,y)\\
&=\sum_{i=1}^N\frac{M_i^h}{|\Omega_i^h(t)|}\int_{\Omega_i^h(t)}|x-x_i^h(t)|\,\diff x
\le Md_h .
\end{aligned}
\end{equation}
By the triangle inequality,
\begin{equation}\label{eq:triangle_W1}
W_1\left(\mu^h(t),\mu(t)\right)
\le Md_h + W_1\left(\nu^h(t),\mu(t)\right).
\end{equation}
For $x$ in the support of $\nu^h(0)$, let $i(x)\in\{1,\dots,N\}$ be the index with $x=x_{i(x)}^h(0)$, and set
\(
X^h(t,x):=x_{i(x)}^h(t).
\)
Then $X^h(t,\cdot)_\#\nu^h(0)=\nu^h(t)$ and $\Phi(t,\cdot)_\#\mu(0)=\mu(t)$.
Hence, for every $\pi_0\in\Pi(\nu^h(0),\mu(0))$, the push-forward of $\pi_0$ by $(x,y)\mapsto (X^h(t,x),\Phi(t,y))$ belongs to $\Pi(\nu^h(t),\mu(t))$, and therefore
\begin{equation}\label{eq:W1_coupling_bound}
W_1\left(\nu^h(t),\mu(t)\right)\le \int_{\Omega\times\Omega} \left|X^h(t,x)-\Phi(t,y)\right|\,\diff\pi_0(x,y).
\end{equation}
Fix $(x,y)$ in the support of $\pi_0$.
(On the torus, we work with continuous representatives in $\R^2$ of $X^h(t,x)$ and $\Phi(t,y)$, chosen so that their Euclidean distance at $t=0$ equals $|x-y|$; the Euclidean distance between the representatives, which is what we estimate below, bounds the distance on the torus from above, which is all we need in \eqref{eq:W1_coupling_bound}.)
By \eqref{eq:forced-lloyd-ode} and $\partial_t\Phi(t,y)=v(t,\Phi(t,y))$, for a.e.\ $t\in[0,T]$,
\begin{align*}
\frac{\diff}{\diff t}\left(X^h(t,x)-\Phi(t,y)\right)
&=v_{i(x)}^h(t)-v(t,x_{i(x)}^h(t))
+v(t,x_{i(x)}^h(t))-v(t,\Phi(t,y))\\
&\quad+\alpha(t)\left(c_{i(x)}(X^h(t))-x_{i(x)}^h(t)\right).
\end{align*}
Since $v(t)$ is Lipschitz, we deduce that
\begin{align*}
\frac{\diff}{\diff t}\left|X^h(t,x)-\Phi(t,y)\right|
&\le L(t)\left|X^h(t,x)-\Phi(t,y)\right|
+\left|v_{i(x)}^h(t)-v(t,x_{i(x)}^h(t))\right|\\
&\quad+\alpha(t)\left|c_{i(x)}(X^h(t))-x_{i(x)}^h(t)\right|
\end{align*}
for a.e.\ $t\in[0,T]$. Gronwall's lemma yields
\begin{align}
\left|X^h(t,x)-\Phi(t,y)\right|
&\le e^{\Lambda(t)}|x-y|
+\int_0^t e^{\Lambda(t)-\Lambda(s)}
\left|v_{i(x)}^h(s)-v(s,x_{i(x)}^h(s))\right|\,\diff s \notag\\
&\quad+\int_0^t e^{\Lambda(t)-\Lambda(s)}\alpha(s)
\left|c_{i(x)}(X^h(s))-x_{i(x)}^h(s)\right|\,\diff s. \label{eq:gronwall_traj}
\end{align}
We integrate \eqref{eq:gronwall_traj} with respect to $\pi_0$.
Since the first marginal of $\pi_0$ is $\nu^h(0)=\sum_{i=1}^N M_i^h\delta_{x_i^h(0)}$, and since the last two terms in \eqref{eq:gronwall_traj} depend on $(x,y)$ only through $i(x)$, we have
\[
\int_{\Omega\times\Omega}\left|v_{i(x)}^h(s)-v(s,x_{i(x)}^h(s))\right|\diff\pi_0(x,y)
=\sum_{i=1}^NM_i^h\left|v_i^h(s)-v(s,x_i^h(s))\right|
=\mathscr E_h(s),
\]
and similarly the integral of the third term equals $B_h(t)$ by \eqref{eq:lloyd_defect_weighted}.
Combining with \eqref{eq:W1_coupling_bound} and taking the infimum over $\pi_0\in\Pi(\nu^h(0),\mu(0))$, we get
\begin{equation}\label{eq:W1_nu_mu_opt}
W_1\left(\nu^h(t),\mu(t)\right)
\le e^{\Lambda(t)}W_1\left(\nu^h(0),\mu(0)\right)
+\int_0^t e^{\Lambda(t)-\Lambda(s)}\mathscr E_h(s)\,\diff s
+B_h(t).
\end{equation}
It remains to estimate $W_1(\nu^h(0),\mu(0))$.
Since $M_i^h=\int_{\Omega_i^h(0)}\rho_0\,\diff y$, the measure
\[
\pi(\diff x,\diff y):=\sum_{i=1}^N\delta_{x_i^h(0)}(\diff x)\otimes\rho_0(y)\mathbf 1_{\Omega_i^h(0)}(y)\,\diff y
\]
has first marginal $\nu^h(0)$ and second marginal $\rho_0\,\diff y=\mu(0)$.
Since $|x_i^h(0)-y|\le d_h$ for $y\in\Omega_i^h(0)$, we get
\begin{equation}\label{eq:initial_coupling}
\begin{aligned}
W_1\left(\nu^h(0),\mu(0)\right)
&\le\int_{\Omega\times\Omega}|x-y|\,\diff\pi(x,y)\\
&=\sum_{i=1}^N\int_{\Omega_i^h(0)}\left|x_i^h(0)-y\right|\rho_0(y)\,\diff y
\le M d_h .
\end{aligned}
\end{equation}
Plugging \eqref{eq:initial_coupling} into \eqref{eq:W1_nu_mu_opt}, and then \eqref{eq:W1_nu_mu_opt} into \eqref{eq:triangle_W1}, we obtain \eqref{eq:W1_conv_bound}.
\end{proof}

We end this section with the following consistency result.

\begin{corollary}\label{cor:weak-consistency}
Under the assumptions of Theorem~\ref{thm:W1-convergence}, for every
$\varphi\in C_c^\infty([0,T)\times\Omega)$ one has
\begin{align*}
&\left|
\int_0^T\!\!\int_\Omega \rho^h(t,x)\left(\partial_t\varphi(t,x)
+v(t,x)\cdot\nabla\varphi(t,x)\right)\,\diff x\,\diff t
+\int_\Omega \rho^h(0,x)\varphi(0,x)\,\diff x
\right|\\
&\qquad\le C\left(
W_1(\mu^h(0),\mu(0)) + \sup_{t\in[0,T]}W_1(\mu^h(t),\mu(t))
\right),
\end{align*}
for some $C>0$ depending only on $\varphi, T, \Omega, v$.

In particular, if $W_1(\mu^h(0),\mu(0))+\sup_{t\in[0,T]}W_1(\mu^h(t),\mu(t))\to 0$, then $\rho^h$ is consistent, in the weak sense, with the continuity equation \eqref{eq:continuity}.
\end{corollary}

\begin{proof}
Since $\rho$ solves \eqref{eq:continuity}, we find
\begin{align*}
&\left|
\int_0^T\!\!\int_\Omega \rho^h\left(\partial_t\varphi+v\cdot\nabla\varphi\right)\,\diff x\,\diff t
+\int_\Omega \rho^h(0,x)\varphi(0,x)\,\diff x
\right|\\
&\qquad\le
\int_0^T\left|
\int_\Omega \left(\partial_t\varphi(t,x)+v(t,x)\cdot\nabla\varphi(t,x)\right)\,
\diff(\mu^h(t)-\mu(t))(x)
\right|\diff t\\
&\qquad\qquad+
\left|
\int_\Omega \varphi(0,x)\,\diff(\mu^h(0)-\mu(0))(x)
\right|.
\end{align*}
Because $\mu^h(t)$ and $\mu(t)$ have the same mass, Kantorovich--Rubinstein duality gives
\[
\left|\int_\Omega f(x)\,\diff(\mu^h(t)-\mu(t))(x)\right|
\le \|\nabla f\|_{L^\infty(\Omega)}\,W_1(\mu^h(t),\mu(t))
\]
for every Lipschitz function $f$.
Applying this with $f=\partial_t\varphi(t,\cdot)+v(t,\cdot)\cdot\nabla\varphi(t,\cdot)$ for each $t$, then integrating in time, and with
$f=\varphi(0,\cdot)$ gives the result.
\end{proof}

\section{Persistence of quasi-uniformity}\label{sec:quasi-uniformity}

In this section we study \eqref{eq:forced-lloyd-ode} with the exact velocity, that is
\begin{equation}\label{eq:exact-ode}
\dot x_i(t)=v(t,x_i(t))+\alpha(t)\left(c_i(X(t))-x_i(t)\right),
\qquad 1\le i\le N .
\end{equation}
The reason is the following.
By Theorem~\ref{thm:W1-convergence}, once the velocity residual $\mathscr E_h$ is small, convergence reduces to two geometric questions, namely that the cell diameters tend to zero and that $B_h(T)$ vanishes.
Equation \eqref{eq:exact-ode} isolates these questions from the approximation of the velocity.

We first construct a \emph{feedback} $\alpha$ which keeps the configuration near the set of minimizers of $\mathscr F$, hence quasi-uniform (Theorem~\ref{thm:alpha-feedback}), and we derive quantitative bounds from it (Corollary~\ref{cor:mesh-bounds}).
We then explain why this is not enough for convergence, and construct in \Secref{subsec:adaptive-energy} an adaptive feedback which gives all the conditions of Theorem~\ref{thm:W1-convergence} (Theorem~\ref{thm:adaptive-energy-convergence}).

Since 
$$
\mathscr F(X)=\int_\Omega\min_{1\le i\le N}|x-x_i|^2\,\diff x
$$ 
(see \Appref{sec:appendix-envelope}), the functional $\mathscr F$ extends continuously to the compact set $\overline\Omega^N$, and has minimizers there.
The global minimization of $\mathscr F$ belongs to the classical theory of optimal quantization and centroidal Vorono\"i tessellations \cite{DuFaberGunzburger1999,GrafLuschgy2000,caglioti2018quantization}.
We do not address the existence of minimizers in the open set $\Omega^N$, nor the geometry of the minimizing set.
We simply assume that the minimizing set is nonempty and compact in $\Omega^N$.

\begin{theorem}\label{thm:alpha-feedback}
Let
\[
\mathscr F^\ast:=\inf_{\Omega^N}\mathscr F,
\qquad
\mathcal M:=\{X\in\Omega^N:\ \mathscr F(X)=\mathscr F^\ast\},
\]
and assume that $\mathcal M$ is nonempty and compact.
Fix $r>0$ and set $\mathcal U_r:=\{X\in\Omega^N:\ \dist(X,\mathcal M)<r\}$, where $r$ is small enough that every $X\in\overline{\mathcal U_r}$ has pairwise distinct generators, and that $\overline{\mathcal U_r}\subset\Omega^N$ when $\Omega$ is a bounded domain.
Assume that $\mathcal U_r$ is quasi-uniform, in the sense that there is $C_1>0$ such that
\begin{equation}\label{eq:H}\tag{H}
\mathscr D(X)\le C_1h
\qquad\text{for every }X\in\mathcal U_r .
\end{equation}
Let $v\in L^1((0,T);C^{0,1}(\Omega;\R^2))$, with $v\cdot n=0$ on $\partial\Omega$ when $\Omega$ is a bounded domain, and set
\[
\delta_r:=\inf\{\mathscr F(X)-\mathscr F^\ast:\ X\in\overline{\mathcal U_r},\ \dist(X,\mathcal M)=r\},
\]
which is positive by compactness.

If $X^0\in\mathcal U_r$ satisfies $\mathscr F(X^0)-\mathscr F^\ast<\delta_r$, then there is a nonnegative $\alpha\in L^1(0,T)$ such that \eqref{eq:exact-ode} with $X(0)=X^0$ has a unique absolutely continuous solution on $[0,T]$, and $X(t)\in\mathcal U_r$ for every $t\in[0,T]$.
In particular $\mathscr D(X(t))\le C_1h$ on $[0,T]$.
\end{theorem}

\begin{proof}
Let $\varepsilon>0$ be such that
\begin{equation}\label{eq:Ebar}
\overline E:=
\mathscr F(X^0)-\mathscr F^\ast
+2\sqrt{|\Omega|}\,\varepsilon\int_0^T\|v(t,\cdot)\|_{L^\infty}\,\diff t
<\delta_r,
\end{equation}
and consider the system
\begin{equation}\label{eq:feedback-ode}
\dot x_i(t)=v(t,x_i(t))+\alpha_\varepsilon(t,X(t))\left(c_i(X(t))-x_i(t)\right),
\qquad 1\le i\le N,
\end{equation}
with $X(0)=X^0$, where
\[
\alpha_\varepsilon(t,X):=
\frac{\sqrt{|\Omega|}\,\|v(t,\cdot)\|_{L^\infty}}
{\sqrt{\mathscr G(X)+\varepsilon^2}} .
\]
By Proposition~\ref{prop:env-gradient} and Remark~\ref{rem:lipschitz}, the centroids $c_i$ and the energy $\mathscr G$ are locally Lipschitz functions of $X$ on the open set of configurations with pairwise distinct generators.
The right-hand side of \eqref{eq:feedback-ode} is therefore locally Lipschitz in $X$ and integrable in $t$, so \eqref{eq:feedback-ode} has a unique absolutely continuous solution $X$ on a maximal interval $[0,T_\ast)$, with $T_\ast\le T$; if $T_\ast<T$, then $X(t)$ leaves every compact subset of the set of configurations with pairwise distinct generators, hence every compact subset of $\mathcal U_r$, as $t\to T_\ast$.
Since $X^0\in\mathcal U_r$ and $\mathcal U_r$ is open,
\[
t_\ast:=\sup\left\{t\in[0,T_\ast):\ X(s)\in\mathcal U_r\ \text{for all }s\in[0,t]\right\}>0 .
\]
We first compute the energy along $X$ on $[0,t_\ast)$.
By Proposition~\ref{prop:env-gradient}, $\mathscr F$ is $C^1$ on the set of configurations with pairwise distinct generators, with 
$$
\nabla_{x_i}\mathscr F(X)=2V_i(X)(x_i-c_i(X)).
$$
By the chain rule and \eqref{eq:feedback-ode}, for a.e.\ $t\in[0,t_\ast)$,
\begin{equation}\label{eq:dFdt}
\frac{\diff}{\diff t}\mathscr F(X(t))
=\mathscr T_v(t)-2\alpha_\varepsilon(t,X(t))\,\mathscr G(X(t)),
\end{equation}
where
\[
\mathscr T_v(t):=2\sum_{i=1}^N V_i(X(t))\left(x_i(t)-c_i(X(t))\right)\cdot v(t,x_i(t)).
\]
We estimate $\mathscr T_v$ in two ways. First, by the Cauchy--Schwarz inequality and $\sum_iV_i=|\Omega|$,
\begin{equation}\label{eq:T_bound_CS}
|\mathscr T_v(t)|\le2\|v(t,\cdot)\|_{L^\infty}\sum_{i=1}^N V_i^{1/2}\cdot V_i^{1/2}|x_i-c_i|\le 2\sqrt{|\Omega|}\,\|v(t,\cdot)\|_{L^\infty}\sqrt{\mathscr G(X(t))},
\end{equation}
where $V_i=V_i(X(t))$, $x_i=x_i(t)$ and $c_i=c_i(X(t))$. Second, we use \eqref{eq:H}, which applies because $X(t)\in\mathcal U_r$ for $t\in[0,t_\ast)$.
By the divergence theorem,
\[
V_i(X)\left(x_i-c_i(X)\right)=\int_{\Omega_i(X)} \left(x_i-x\right)\,\diff x
=-\frac12\int_{\partial\Omega_i(X)}|x-x_i|^2\,n_i(x)\,\diff s,
\]
where $n_i$ is the outer normal to $\Omega_i(X)$, and therefore
\begin{equation}\label{eq:Tv_faces}
\mathscr T_v(t)=-\sum_{i=1}^N\int_{\partial\Omega_i(X(t))}|x-x_i(t)|^2\,v(t,x_i(t))\cdot n_i(x)\,\diff s .
\end{equation}
Let $\Gamma_{ij}(t):=\partial\Omega_i(X(t))\cap\partial\Omega_j(X(t))$ be the face between two neighboring cells, oriented by the normal $n_{ij}$ from $\Omega_i(X(t))$ to $\Omega_j(X(t))$, and let $\Gamma_i^\partial(t):=\partial\Omega_i(X(t))\cap\partial\Omega$.
On the torus, $\Gamma_i^\partial(t)$ is empty, the cells are the convex polygons $P_i(t)$ of \Secref{subsec:adaptive-energy}, $\Gamma_{ij}(t)$ is the union of the faces between $P_i(t)$ and the translates of $P_j(t)$, and the faces between $P_i(t)$ and its own translates cancel in \eqref{eq:Tv_faces}, since they come in pairs with opposite normals.
On $\Gamma_{ij}(t)$ we have $n_i=-n_j=n_{ij}$ and $|x-x_i(t)|=|x-x_j(t)|$, and on $\Gamma_i^\partial(t)$ we have $v(t,x)\cdot n(x)=0$.
Hence \eqref{eq:Tv_faces} can be rewritten as
\begin{equation}\label{eq:Tv_faces2}
\begin{aligned}
\mathscr T_v(t)
&=-\sum_{i<j}\int_{\Gamma_{ij}(t)} |x-x_i(t)|^2\left(v(t,x_i(t))-v(t,x_j(t))\right)\cdot n_{ij}(x)\,\diff s\\
&\quad-\sum_{i=1}^N\int_{\Gamma_i^\partial(t)} |x-x_i(t)|^2\left(v(t,x_i(t))-v(t,x)\right)\cdot n(x)\,\diff s .
\end{aligned}
\end{equation}
Now let $x\in\Gamma_{ij}(t)$. Since $x$ belongs to the closures of both $\Omega_i(X(t))$ and $\Omega_j(X(t))$, \eqref{eq:H} gives
\[
|x-x_i(t)|\le C_1h,
\qquad
|x_i(t)-x_j(t)|\le|x_i(t)-x|+|x-x_j(t)|\le2C_1h,
\]
and similarly $|x-x_i(t)|\le C_1h$ for $x\in\Gamma_i^\partial(t)$.
Since $v(t,\cdot)$ is $L(t)$-Lipschitz, we deduce that
\begin{equation}\label{eq:v_faces}
\begin{aligned}
\left|v(t,x_i(t))-v(t,x_j(t))\right|&\le 2C_1L(t)h
&&\text{on }\Gamma_{ij}(t),\\
\left|v(t,x_i(t))-v(t,x)\right|&\le C_1L(t)h
&&\text{on }\Gamma_i^\partial(t).
\end{aligned}
\end{equation}
The faces $\Gamma_{ij}(t)$, $j\ne i$, lie on the boundary of the convex Vorono\"i polygon of $x_i(t)$ in $\R^2$, and inside the ball $B(x_i(t),C_1h)$; their total length is thus at most the perimeter of this ball, namely $2\pi C_1h$.
Moreover $\sum_i|\Gamma_i^\partial(t)|\le|\partial\Omega|$.
Plugging \eqref{eq:v_faces} into \eqref{eq:Tv_faces2} and using $Nh^2=|\Omega|$ and $h\le\sqrt{|\Omega|}$, we find
\begin{equation}\label{eq:T_bound_h2}
\begin{aligned}
|\mathscr T_v(t)|
&\le N\cdot 2\pi C_1h\cdot(C_1h)^2\cdot 2C_1L(t)h
+|\partial\Omega|\cdot(C_1h)^2\cdot C_1L(t)h\\
&\le C_2L(t)h^2,
\end{aligned}
\end{equation}
where $C_2:=4\pi C_1^4|\Omega|+C_1^3|\partial\Omega|\sqrt{|\Omega|}$.

We now show that $t_\ast=T_\ast=T$.
By \eqref{eq:dFdt}, \eqref{eq:T_bound_CS} and the definition of $\alpha_\varepsilon$, for a.e.\ $t\in[0,t_\ast)$,
\begin{equation}\label{eq:dFdt_eps}
\begin{aligned}
\frac{\diff}{\diff t}\mathscr F(X(t))
&\le 2\sqrt{|\Omega|}\,\|v(t,\cdot)\|_{L^\infty}
\left(\sqrt{\mathscr G(X(t))}-\frac{\mathscr G(X(t))}{\sqrt{\mathscr G(X(t))+\varepsilon^2}}\right)\\
&\le 2\sqrt{|\Omega|}\,\varepsilon\,\|v(t,\cdot)\|_{L^\infty}.
\end{aligned}
\end{equation}
Integrating \eqref{eq:dFdt_eps} on $[0,t]$ and using \eqref{eq:Ebar}, we get
\begin{equation}\label{eq:F_trapped}
\mathscr F(X(t))-\mathscr F^\ast\le\overline E<\delta_r
\qquad\text{for all }t\in[0,t_\ast).
\end{equation}
Suppose that $t_\ast<T_\ast$. Then $\dist(X(t_\ast),\mathcal M)=r$, and the definition of $\delta_r$ gives $\mathscr F(X(t_\ast))-\mathscr F^\ast\ge\delta_r$.
On the other hand, letting $t\to t_\ast$ in \eqref{eq:F_trapped} and using the continuity of $\mathscr F$, we get $\mathscr F(X(t_\ast))-\mathscr F^\ast\le\overline E<\delta_r$, a contradiction.
Hence $t_\ast=T_\ast$, and \eqref{eq:F_trapped} holds on $[0,T_\ast)$.
Consider now the set
\[
K:=\left\{X\in\overline{\mathcal U_r}:\ \mathscr F(X)-\mathscr F^\ast\le\overline E\right\}.
\]
It is compact, and it is contained in $\mathcal U_r$, because $\mathscr F-\mathscr F^\ast\ge\delta_r>\overline E$ on $\overline{\mathcal U_r}\setminus\mathcal U_r$.
By \eqref{eq:F_trapped}, $X(t)\in K$ for all $t\in[0,T_\ast)$, so $X$ does not leave every compact subset of $\mathcal U_r$, and thus $T_\ast=T$.
Since $|\dot X|$ is integrable on $[0,T)$, $X$ extends continuously to $[0,T]$, with $X(T)\in K$ because $K$ is closed.

Finally, set $\alpha(t):=\alpha_\varepsilon(t,X(t))$ for $t\in[0,T]$. Then $\alpha\ge0$,
\[
\int_0^T\alpha(t)\,\diff t
\le\frac{\sqrt{|\Omega|}}{\varepsilon}
\int_0^T\|v(t,\cdot)\|_{L^\infty}\,\diff t<\infty,
\]
and $X$ is the unique absolutely continuous solution of \eqref{eq:exact-ode} with this $\alpha$ and $X(0)=X^0$, again because the right-hand side of \eqref{eq:exact-ode} is locally Lipschitz in $X$ and integrable in $t$.
Thus $X(t)\in\mathcal U_r$ for all $t\in[0,T]$.
\end{proof}

The hypothesis \eqref{eq:H} holds for $r=O(h)$.
Every $Y\in\mathcal M$ satisfies $\sup_{x\in\Omega}\min_i|x-y_i|\le Ch$ with $C$ independent of $h$, by \cite[Proposition~3.1]{GrafLuschgyPages2012} for a bounded convex domain and by \cite[Theorem~1(ii)]{Gruber2004} for a flat torus.
Moving each generator by at most $r$ increases this covering radius by at most $r$, and a cell has diameter at most twice the covering radius, so $\mathscr D(X)\le2(Ch+r)$ for $X\in\mathcal U_r$.

We now make Theorem~\ref{thm:alpha-feedback} quantitative, in view of Theorem~\ref{thm:W1-convergence}.

\begin{corollary}\label{cor:mesh-bounds}
Under the assumptions of Theorem~\ref{thm:alpha-feedback}, with the constant $C_1$ in \eqref{eq:H} independent of $h$, let $X^h$ and $\alpha_h$ be the solutions and coefficients given by Theorem~\ref{thm:alpha-feedback} for radii $r_h$, and set $A_h:=\int_0^T\alpha_h(t)\,\diff t$.
There exists a constant $C_3>0$, depending only on $C_1$ and $\Omega$, with the following property.
If $r_h\le C_3h$, then the following hold for every $t\in[0,T]$.
\begin{enumerate}[label=\textup{(\roman*)}]
\item the generators stay separated, with $|x_i^h(t)-x_j^h(t)|\ge2C_3h$ for $i\ne j$;
\item the cell volumes $V_i^h(t):=|\Omega_i^h(t)|$ satisfy $V_i^h(t)\ge h^2/C_4$ and $V_i^h(t)\ge V_i^h(0)/C_4$;
\item $\int_0^T\alpha_h(t)\mathscr G(X^h(t))\,\diff t\le C_4h^2$;
\item if $\rho_0\in L^p(\Omega)$ for some $1<p\le\infty$, then 
\begin{equation}\label{eq:Bh-Lp}
B_h(T)\le C_5\,h\,A_h^{1/\min\{p,2\}} .
\end{equation}
\end{enumerate}
The constant $C_4$ depends only on $C_1$, $\Omega$ and $\Lambda(T)$, and $C_5$ depends also on $p$ and $\|\rho_0\|_{L^p}$.
\end{corollary}

\begin{proof}
There exists $c_\Omega>0$ such that
\begin{equation}\label{eq:cone}
|B(z,s)\cap\Omega|\ge c_\Omega s^2
\qquad\text{for all }z\in\Omega\text{ and }0<s\le\operatorname{diam}\Omega .
\end{equation}
For a bounded Lipschitz domain, this is the uniform interior cone condition.
For a flat torus, let $r_0$ be half the length of the shortest nonzero period of $\Omega$. For $0<s\le r_0$ the ball $B(z,s)$ is isometric to a Euclidean disc, so $|B(z,s)|=\pi s^2$, while for $r_0<s\le\operatorname{diam}\Omega$ we have $|B(z,s)|\ge\pi r_0^2$; thus $c_\Omega:=\pi(r_0/\operatorname{diam}\Omega)^2$ suffices.

We first prove that the generators of any $Y=(y_1,\dots,y_N)\in\mathcal M$ are separated.
Since $Y\in\mathcal M\subset\mathcal U_{r_h}$, \eqref{eq:H} gives $\operatorname{diam}\Omega_k(Y)\le C_1h$, and thus $V_k(Y)\le\pi C_1^2h^2$, for all $k$.
Let $\delta:=\min_{k\ne l}|y_k-y_l|=|y_i-y_j|$.
Removing $y_j$ gives a configuration $Y'$ with $N-1$ generators, whose energy $\mathscr F(Y')=\int_\Omega\min_{k\ne j}|x-y_k|^2\,\diff x$ satisfies
\begin{equation}\label{eq:removal}
\begin{aligned}
\mathscr F(Y')
&\le\mathscr F(Y)+\int_{\Omega_j(Y)}\left(|x-y_i|^2-|x-y_j|^2\right)\diff x\\
&\le\mathscr F(Y)+\pi C_1^2h^2\,\delta\left(2C_1h+\delta\right),
\end{aligned}
\end{equation}
where we used $\min_{k\ne j}|x-y_k|^2\le|x-y_i|^2$, as well as $|x-y_i|\le|x-y_j|+\delta\le C_1h+\delta$ and $|x-y_j|\le C_1h$ on $\Omega_j(Y)$.
Let $R:=h/\sqrt{2\pi}$. Since $N\pi R^2=|\Omega|/2<|\Omega|$, the balls $B(y_k,R)$, $1\le k\le N$, do not cover $\Omega$, and we may fix $z\in\Omega$ with $|z-y_k|\ge R$ for all $k$.
Let $Y''$ be the configuration obtained by adding $z$ to $Y'$.
For $x\in B(z,R/4)$ we have $\min_{k\ne j}|x-y_k|\ge\min_k|x-y_k|\ge3R/4$ and $|x-z|\le R/4$, whence, by \eqref{eq:cone},
\begin{equation}\label{eq:insertion}
\begin{aligned}
\mathscr F(Y'')
&\le\mathscr F(Y')-\int_{B(z,R/4)\cap\Omega}\left(\min_{k\ne j}|x-y_k|^2-|x-z|^2\right)\diff x\\
&\le\mathscr F(Y')-\frac{R^2}{2}\,|B(z,R/4)\cap\Omega|
\le\mathscr F(Y')-\frac{c_\Omega R^4}{32} .
\end{aligned}
\end{equation}
Since $Y''\in\Omega^N$ has $N$ distinct generators, $\mathscr F(Y'')\ge\mathscr F^\ast=\mathscr F(Y)$.
Combining this with \eqref{eq:removal} and \eqref{eq:insertion}, and using $R^4=h^4/(4\pi^2)$, we find
\begin{equation}\label{eq:sep_ineq}
\frac{c_\Omega h^4}{128\pi^2}\le\pi C_1^2h^2\,\delta\left(2C_1h+\delta\right).
\end{equation}
If $\delta\le h$, then $\delta(2C_1h+\delta)\le(2C_1+1)h\delta$, and \eqref{eq:sep_ineq} yields $\delta\ge c_\Omega h/(128\pi^3C_1^2(2C_1+1))$. Therefore
\begin{equation}\label{eq:sep_min}
\delta\ge4C_3h,
\qquad
C_3:=\frac14\min\left\{1,\frac{c_\Omega}{128\pi^3C_1^2(2C_1+1)}\right\}.
\end{equation}
Now let $r_h\le C_3h$ and fix $t\in[0,T]$.
By Theorem~\ref{thm:alpha-feedback}, $\dist(X^h(t),\mathcal M)<r_h$, so there is $Y\in\mathcal M$ with $|x_i^h(t)-y_i|<r_h$ for all $i$.
By \eqref{eq:sep_min},
\[
|x_i^h(t)-x_j^h(t)|\ge|y_i-y_j|-2r_h\ge4C_3h-2C_3h=2C_3h
\qquad\text{for }i\ne j,
\]
which is \textup{(i)}.
Consequently, if $|x-x_i^h(t)|<C_3h$, then $|x-x_j^h(t)|>C_3h>|x-x_i^h(t)|$ for all $j\ne i$, that is, $B(x_i^h(t),C_3h)\cap\Omega\subset\Omega_i^h(t)$.
By \eqref{eq:cone} and \eqref{eq:H},
\begin{equation}\label{eq:vol_bounds}
V_i^h(t)\ge c_\Omega C_3^2h^2
\ge\frac{c_\Omega C_3^2}{\pi C_1^2}\,V_i^h(0),
\end{equation}
which gives \textup{(ii)} for $C_4\ge\max\{1,\pi C_1^2\}/(c_\Omega C_3^2)$.
For \textup{(iii)}, we integrate \eqref{eq:dFdt} on $[0,T]$ and use $\mathscr F\ge0$, \eqref{eq:T_bound_h2}, and $\mathscr F(X^h(0))\le\sum_iV_i^h(0)\,\mathscr D(X^h(0))^2\le|\Omega|C_1^2h^2$, to get
\[
2\int_0^T\alpha_h(t)\mathscr G(X^h(t))\,\diff t
\le\mathscr F(X^h(0))+\int_0^T|\mathscr T_v(t)|\,\diff t
\le\left(C_1^2|\Omega|+C_2\Lambda(T)\right)h^2 .
\]
This gives \textup{(iii)} for $C_4\ge(C_1^2|\Omega|+C_2\Lambda(T))/2$.

We finally prove \textup{(iv)}.
Let $\rho^h$ be the Vorono\"i histogram of Theorem~\ref{thm:W1-convergence} and let
\[
D_h(t,x):=\sum_{i=1}^N\left|c_i(X^h(t))-x_i^h(t)\right|\mathbf 1_{\Omega_i^h(t)}(x).
\]
Then
\begin{equation}\label{eq:Dh}
\sum_{i=1}^N M_i^h\left|c_i(X^h(t))-x_i^h(t)\right|
=\int_\Omega\rho^h(t,x)D_h(t,x)\,\diff x,
\qquad
\|D_h(t)\|_{L^2}^2=\mathscr G(X^h(t)),
\end{equation}
and $\|D_h(t)\|_{L^\infty}\le C_1h$, since $|c_i-x_i|\le\operatorname{diam}\Omega_i\le C_1h$ by \eqref{eq:H}.
Let $1<p\le2$ and $q:=p/(p-1)\ge2$.
By \textup{(ii)} and Jensen's inequality,
\begin{equation}\label{eq:rho_Lp}
\begin{aligned}
\|\rho^h(t)\|_{L^p}^p
=\sum_{i=1}^N V_i^h(t)^{1-p}\,(M_i^h)^p
&\le C_4^{p-1}\sum_{i=1}^N V_i^h(0)\left(\frac{M_i^h}{V_i^h(0)}\right)^p\\
&\le C_4^{p-1}\sum_{i=1}^N\int_{\Omega_i^h(0)}\rho_0^p\,\diff x
=C_4^{p-1}\|\rho_0\|_{L^p}^p .
\end{aligned}
\end{equation}
By H\"older's inequality, the interpolation inequality $\|D_h\|_{L^q}\le\|D_h\|_{L^\infty}^{1-2/q}\|D_h\|_{L^2}^{2/q}$, \eqref{eq:Dh} and \eqref{eq:rho_Lp},
\begin{equation}\label{eq:rhoD}
\begin{aligned}
\int_\Omega\rho^h(t,x)D_h(t,x)\,\diff x
&\le\|\rho^h(t)\|_{L^p}\|D_h(t)\|_{L^q}\\
&\le C_4^{1/q}\|\rho_0\|_{L^p}(C_1h)^{1-2/q}\mathscr G(X^h(t))^{1/q} .
\end{aligned}
\end{equation}
Plugging \eqref{eq:Dh} and \eqref{eq:rhoD} into the definition \eqref{eq:lloyd_defect_weighted} of $B_h(T)$, then using H\"older's inequality with respect to the measure $\alpha_h(t)\,\diff t$ with exponents $p$ and $q$, and finally \textup{(iii)}, we get
\begin{align*}
B_h(T)
&\le e^{\Lambda(T)}C_4^{1/q}\|\rho_0\|_{L^p}(C_1h)^{1-2/q}\int_0^T\alpha_h(t)\,\mathscr G(X^h(t))^{1/q}\,\diff t\\
&\le e^{\Lambda(T)}C_4^{1/q}\|\rho_0\|_{L^p}(C_1h)^{1-2/q}A_h^{1/p}\left(\int_0^T\alpha_h(t)\mathscr G(X^h(t))\,\diff t\right)^{1/q}\\
&\le e^{\Lambda(T)}C_4^{2/q}C_1^{1-2/q}\|\rho_0\|_{L^p}\,h\,A_h^{1/p}.
\end{align*}
This is \eqref{eq:Bh-Lp} for $1<p\le2$.
For $2<p\le\infty$, we have $\|\rho_0\|_{L^2}\le|\Omega|^{1/2-1/p}\|\rho_0\|_{L^p}$ since $\Omega$ is bounded, and \eqref{eq:Bh-Lp} follows from the case $p=2$.
\end{proof}

Corollary~\ref{cor:mesh-bounds} gives quasi-uniformity and controls the cell volumes, but our estimates do not suffice for convergence.
Indeed, \eqref{eq:Bh-Lp} would give $B_h(T)\to0$ if $A_h=o(h^{-\min\{p,2\}})$, but we have no such bound on $A_h$, and we do not know how to prove that $B_h(T)\to0$ for this feedback.
The next section introduces a different feedback for which the energy dissipation controls the total correction $A_h$ and yields convergence.

\section{A more general feedback which retains consistency}\label{subsec:adaptive-energy}

 
We now allow any $\alpha_h$ which is locally Lipschitz on the set of configurations with pairwise distinct generators and satisfies $0\leq \alpha_h(X)\leq \mathscr G(X)/h^{5/2}$.
The energy $\mathscr F$ is then allowed to grow with the deformation caused by $v$, but it stays of order $h^2$, and the dissipation $\int_0^T\alpha_h\mathscr G\,\diff t$ controls both the centroid displacements and the total correction $A_h$.
We expect this regime to be the most interesting one in view of numerical applications.

Throughout this section, $\Omega=\R^2/\mathcal L$ is a flat torus as in \Secref{sec:moving}.
Given representatives $x_i\in\R^2$ of the generators,  consider the $\mathcal L$-periodic set $\{x_i+k:\ 1\le i\le N,\ k\in\mathcal L\}\subset\R^2$ and its Vorono\"i tessellation of $\R^2$.
The cell of $x_i$ in this tessellation is the bounded convex polygon
\[
P_i:=\left\{y\in\R^2:\ |y-x_i|\le|y-x_j-k|\ \text{for all }1\le j\le N,\ k\in\mathcal L\right\},
\]
that is, the set of $y\in\R^2$ whose Euclidean distance to $x_i$ equals $\min_j|y-x_j|$; the cell of $x_j+k$ is $P_j+k$, and $\Omega_i$ is the image of $P_i$ in $\Omega$, up to a set of measure zero. In particular
\[
V_i=|P_i|,
\qquad c_i-x_i=\frac1{V_i}\int_{P_i}(y-x_i)\,\diff y,
\]
and the displacement $c_i-x_i$ does not depend on the choice of the representative $x_i$.

\begin{theorem}\label{thm:adaptive-energy-convergence}
Let $v\in L^1((0,T);C^{0,1}(\Omega;\R^2))$.
For each $N\ge2$, set $h=\sqrt{|\Omega|/N}$ and take distinct initial generators $X^h(0)$ with
\begin{equation}\label{eq:adaptive-initial-energy}
\mathscr F(X^h(0))\le C_1h^2,
\end{equation}
where $C_1$ does not depend on $h$.
Let $\alpha_h$ be locally Lipschitz on the open set of configurations with pairwise distinct generators, with
$0\leq \alpha_h(X)\leq \mathscr G(X)/h^{5/2}$,
and consider
\begin{equation}\label{eq:adaptive-energy-feedback}
\dot x_i^h(t)=v(t,x_i^h(t))+\alpha_h(X^h(t))\left(c_i(X^h(t))-x_i^h(t)\right),
\qquad 1\le i\le N.
\end{equation}
Then this Cauchy problem has a unique absolutely continuous solution on $[0,T]$, the generators remain distinct, and $\alpha_h(t):=\alpha_h(X^h(t))$ is nonnegative and bounded for each fixed $h$.
There is a constant $C_2$, depending only on $C_1$, $T$, $\Omega$ and $\Lambda(T)$, such that
\begin{align}
\sup_{t\in[0,T]}\mathscr F(X^h(t))
+\int_0^T\alpha_h(t)\mathscr G(X^h(t))\,\diff t
&\le C_2 h^2,\label{eq:adaptive-energy-dissipation}\\
d_h=\sup_{t\in[0,T]}\mathscr D(X^h(t))
&\le C_2 h^{1/2},\label{eq:adaptive-diameter}\\
A_h=\int_0^T\alpha_h(t)\,\diff t
&\le C_2 h^{-1/4},\label{eq:adaptive-alpha-budget}\\
\int_0^T\left(\frac{\mathscr G(X^h(t))}{h^2}\right)^2\,\diff t
&\le C_2 h^{1/2}\quad\text{if }\alpha_h=\mathscr G/h^{5/2}.\label{eq:adaptive-centroid-error}
\end{align}
Moreover for every $\rho_0\in L^1(\Omega)$ with $\rho_0\ge0$, the Vorono\"i histogram $\mu^h$ of Theorem~\ref{thm:W1-convergence} and the solution $\mu$ of \eqref{eq:continuity} satisfy, for all $h\le1$,
\begin{equation}\label{eq:adaptive-transport-convergence}
B_h(T)+\sup_{t\in[0,T]}W_1(\mu^h(t),\mu(t))
\le C_2\|\rho_0\|_{L^1}h^{1/4}.
\end{equation}
\end{theorem}


\begin{proof}
In the proof $C$ denotes a constant depending only on $C_1$, $T$, $\Omega$ and $\Lambda(T)$, whose value may change from one line to the next.
Fix $h$.
By the local Lipschitz assumption on $\alpha_h$ and Remark~\ref{rem:lipschitz}, the right-hand side of \eqref{eq:adaptive-energy-feedback} is locally Lipschitz in $X$ on the open set of configurations with pairwise distinct generators, with locally integrable bounds in $t$.
Hence \eqref{eq:adaptive-energy-feedback} has a unique absolutely continuous solution $X^h$ on a maximal interval $[0,T_\ast)$, with $T_\ast\le T$, and if $T_\ast<T$ then $\min_{i\ne j}|x_i^h(t)-x_j^h(t)|\to0$ as $t\to T_\ast$, since $\Omega^N$ is compact.
Since $|c_i-x_i|\le\operatorname{diam}\Omega$ and $\sum_iV_i=|\Omega|$, we have $\mathscr G\le|\Omega|\operatorname{diam}(\Omega)^2$, whence
\begin{equation}\label{eq:alpha_bounded}
0\le\alpha_h(t)\le\overline\alpha_h:=|\Omega|\operatorname{diam}(\Omega)^2h^{-5/2}
\qquad\text{for }t\in[0,T_\ast).
\end{equation}
Fix $i\ne j$, representatives $x_i^h(t),x_j^h(t)\in\R^2$ depending continuously on $t$, and $k\in\mathcal L$.
Set $w(t):=x_i^h(t)-x_j^h(t)-k$.
The cells $P_i$ and $P_j+k$ lie on opposite sides of the bisector of $x_i^h(t)$ and $x_j^h(t)+k$, and so do their centroids $c_i$ and $c_j+k$; hence
\begin{equation}\label{eq:bisector}
w(t)\cdot\left(c_i(X^h(t))-c_j(X^h(t))-k\right)\ge0 .
\end{equation}
By \eqref{eq:adaptive-energy-feedback}, and since $(c_i-x_i^h)-(c_j-x_j^h)=(c_i-c_j-k)-w$, we have for a.e.\ $t\in[0,T_\ast)$
\begin{align*}
\frac12\frac{\diff}{\diff t}|w(t)|^2
&=w(t)\cdot\left(v(t,x_i^h(t))-v(t,x_j^h(t)+k)\right)\\
&\quad+\alpha_h(t)\,w(t)\cdot\left(c_i(X^h(t))-c_j(X^h(t))-k\right)
-\alpha_h(t)|w(t)|^2,
\end{align*}
where we used the periodicity of $v$.
Since $|v(t,x_i^h(t))-v(t,x_j^h(t)+k)|\le L(t)|w(t)|$, we deduce from \eqref{eq:bisector} and \eqref{eq:alpha_bounded} that
\[
\frac12\frac{\diff}{\diff t}|w(t)|^2\ge-\left(L(t)+\overline\alpha_h\right)|w(t)|^2,
\]
whence $|w(t)|\ge|w(0)|e^{-\Lambda(t)-\overline\alpha_ht}$. Since $|w(0)|\ge|x_i^h(0)-x_j^h(0)|$, taking the minimum over $k\in\mathcal L$ we get
\[
\left|x_i^h(t)-x_j^h(t)\right|\ge\left|x_i^h(0)-x_j^h(0)\right|e^{-\Lambda(T)-\overline\alpha_hT}>0
\qquad\text{for all }t\in[0,T_\ast).
\]
Therefore $T_\ast=T$, so that the solution exists on $[0,T]$, the generators remain distinct, and $\alpha_h$ is bounded by \eqref{eq:alpha_bounded}.

By Proposition~\ref{prop:env-gradient} on the torus, $\mathscr F$ is $C^1$ on the set of configurations with pairwise distinct generators, with $\nabla_{x_i}\mathscr F=2V_i(x_i-c_i)$.
By the chain rule and \eqref{eq:adaptive-energy-feedback}, for a.e.\ $t\in[0,T]$,
\begin{equation}\label{eq:adaptive-energy-identity}
\frac{\diff}{\diff t}\mathscr F(X^h(t))
=\mathscr T_v(t)-2\alpha_h(t)\mathscr G(X^h(t)),
\end{equation}
where, denoting by $P_i(t)$ the cell of $x_i^h(t)$ in the periodic tessellation,
\begin{align*}
\mathscr T_v(t)&:=2\sum_{i=1}^N V_i(X^h(t))\left(x_i^h(t)-c_i(X^h(t))\right)\cdot v(t,x_i^h(t))\\
&=2\sum_{i=1}^N\int_{P_i(t)}\left(x_i^h(t)-y\right)\cdot v(t,x_i^h(t))\,\diff y .
\end{align*}
We write $v(t,x_i^h(t))=(v(t,x_i^h(t))-v(t,y))+v(t,y)$.
For $y\in P_i(t)$, the Euclidean distance $|y-x_i^h(t)|$ equals $\min_j|y-x_j^h(t)|$, so $|v(t,x_i^h(t))-v(t,y)|\le L(t)|x_i^h(t)-y|$, and therefore
\begin{align}
2\sum_{i=1}^N\int_{P_i(t)}\bigl(x_i^h(t)-y\bigr)\cdot\bigl(v(t,x_i^h(t))-v(t,y)\bigr)\diff y
&\le2L(t)\sum_{i=1}^N\int_{P_i(t)}\left|x_i^h(t)-y\right|^2\diff y\notag\\
&=2L(t)\mathscr F(X^h(t)).\label{eq:Tv_part1}
\end{align}
Then let $f(t,y):=\min_j|y-x_j^h(t)|^2$.
The function $f(t,\cdot)$ is $\mathcal L$-periodic and Lipschitz, $\nabla_yf(t,y)=2(y-x_i^h(t))$ for a.e.\ $y\in P_i(t)$, and $\int_\Omega f(t,y)\,\diff y=\mathscr F(X^h(t))$.
Integrating by parts and using $|\operatorname{div}v(t,\cdot)|\le2L(t)$, we get
\begin{align}
2\sum_{i=1}^N\int_{P_i(t)}\left(x_i^h(t)-y\right)\cdot v(t,y)\,\diff y
&=-\int_\Omega\nabla_yf(t,y)\cdot v(t,y)\,\diff y\notag\\
&=\int_\Omega f(t,y)\operatorname{div}v(t,y)\,\diff y\notag\\
&\le2L(t)\mathscr F(X^h(t)).\label{eq:Tv_part2}
\end{align}
Combining \eqref{eq:adaptive-energy-identity}, \eqref{eq:Tv_part1} and \eqref{eq:Tv_part2}, we find
\begin{equation}\label{eq:adaptive-global-energy-inequality}
\frac{\diff}{\diff t}\mathscr F(X^h(t))+2\alpha_h(t)\mathscr G(X^h(t))
\le4L(t)\mathscr F(X^h(t))
\qquad\text{for a.e.\ }t\in[0,T].
\end{equation}
Multiplying \eqref{eq:adaptive-global-energy-inequality} by $e^{-4\Lambda(t)}$, integrating on $[0,t]$ and using \eqref{eq:adaptive-initial-energy}, we get
\[
e^{-4\Lambda(t)}\mathscr F(X^h(t))
+2\int_0^t e^{-4\Lambda(s)}\alpha_h(s)\mathscr G(X^h(s))\,\diff s
\le\mathscr F(X^h(0))\le C_1h^2
\]
for all $t\in[0,T]$,
whence
\begin{equation}\label{eq:energy_bounds}
\sup_{t\in[0,T]}\mathscr F(X^h(t))\le e^{4\Lambda(T)}C_1h^2,
\qquad
\int_0^T\alpha_h(t)\mathscr G(X^h(t))\,\diff t\le\frac12e^{4\Lambda(T)}C_1h^2 .
\end{equation}
This proves \eqref{eq:adaptive-energy-dissipation}.

We now bound the cell diameters by the energy.
Inspired by the arguments in \cite[Section~2.5]{KriegSonnleitner2024} and \cite[proof of Proposition~2.1]{ElNmeirLuschgyPages2022}, we proceed as follows.
Let $X\in\Omega^N$ have pairwise distinct generators, let $R(X):=\max_{y\in\Omega}\min_j|y-x_j|$, and let $y_\ast\in\Omega$ be a point where the maximum is attained.
Since $y_\ast$ belongs to the closure of some cell $\Omega_i(X)$, we have $R(X)=|y_\ast-x_i|\le\operatorname{diam}\Omega_i(X)\le\mathscr D(X)$.
Since $y\mapsto\min_j|y-x_j|$ is $1$-Lipschitz, $\min_j|y-x_j|\ge R(X)/2$ for $y\in B(y_\ast,R(X)/2)$, whence
\begin{equation}\label{eq:F_vs_R}
\mathscr F(X)
=\int_\Omega\min_j|y-x_j|^2\,\diff y
\ge\frac{R(X)^2}{4}\left|B\left(y_\ast,R(X)/2\right)\right|
\ge\frac{c_\Omega}{16}R(X)^4,
\end{equation}
where $c_\Omega$ is the constant in \eqref{eq:cone}.
Furthermore, for $y\in P_i$ the Euclidean distance $|y-x_i|$ equals $\min_j|y-x_j|\le R(X)$, so $P_i\subset B(x_i,R(X))$, and $c_i\in P_i$ by convexity. Hence
\begin{equation}\label{eq:adaptive-radius-displacement}
\operatorname{diam}\Omega_i(X)\le2R(X),
\qquad |c_i(X)-x_i|\le R(X)\le\mathscr D(X).
\end{equation}
By \eqref{eq:F_vs_R} and \eqref{eq:energy_bounds}, one obtains
\begin{equation}\label{eq:rf}
R(X^h(t))\le(16e^{4\Lambda(T)}C_1/c_\Omega)^{1/4}h^{1/2}
\end{equation}
 for all $t\in[0,T]$, and \eqref{eq:adaptive-diameter} follows from \eqref{eq:adaptive-radius-displacement}.

Since $\alpha_h(t)^2\le h^{-5/2}\alpha_h(t)\mathscr G(X^h(t))$, Cauchy--Schwarz and \eqref{eq:energy_bounds} give
\begin{equation}\label{eq:alpha_squared}
A_h^2\le T\int_0^T\alpha_h(t)^2\,\diff t
\le T h^{-5/2}\int_0^T\alpha_h(t)\mathscr G(X^h(t))\,\diff t
\le Ch^{-1/2}.
\end{equation}
This proves \eqref{eq:adaptive-alpha-budget} for every feedback in the stated class.
For the particular choice $\alpha_h(X)=\mathscr G(X)/h^{5/2}$, which is locally Lipschitz by Remark~\ref{rem:lipschitz}, one also has
\[
\int_0^T\mathscr G(X^h(t))^2\,\diff t
=h^{5/2}\int_0^T\alpha_h(t)\mathscr G(X^h(t))\,\diff t
\le Ch^{9/2}.
\]
Dividing by $h^4$ gives \eqref{eq:adaptive-centroid-error}.

Finally, for every feedback in the stated class, we apply Theorem~\ref{thm:W1-convergence} to $X^h$, with $\alpha=\alpha_h$ and $v_i^h(t)=v(t,x_i^h(t))$.
Its assumptions all hold, since $X^h$ is absolutely continuous, $\alpha_h\in L^1(0,T)$ is nonnegative, the generators remain distinct, and the velocity residual $\mathscr E_h$ vanishes identically.
By \eqref{eq:lloyd_defect_weighted}, \eqref{eq:adaptive-radius-displacement}, \eqref{eq:adaptive-diameter} and \eqref{eq:adaptive-alpha-budget}, with $M=\|\rho_0\|_{L^1}$,
\begin{align*}
B_h(T)
&\le e^{\Lambda(T)}\int_0^T\alpha_h(t)\sum_{i=1}^NM_i^h\left|c_i(X^h(t))-x_i^h(t)\right|\diff t\\
&\le e^{\Lambda(T)}M\,d_h\,A_h
\le CMh^{1/2}h^{-1/4}=CMh^{1/4}.
\end{align*}
Plugging this into \eqref{eq:W1_conv_bound} and using \eqref{eq:adaptive-diameter} once more, we obtain
\[
\sup_{t\in[0,T]}W_1\left(\mu^h(t),\mu(t)\right)
\le M\left(1+e^{\Lambda(T)}\right)d_h+B_h(T)
\le CMh^{1/2}+CMh^{1/4}
\]
for all $h$, and hence $\sup_{t\in[0,T]}W_1(\mu^h(t),\mu(t))\le CMh^{1/4}$ for $h\le1$, which is \eqref{eq:adaptive-transport-convergence}.
\end{proof}

\begin{remark}
The critical estimate in the proof is \eqref{eq:rf}, which comes from the bound \eqref{eq:F_vs_R}.
That bound comes for free for any Vorono\"i tessellation, and does not use the structure of the Lloyd algorithm; we therefore expect it to be suboptimal.
\end{remark}

\section{One-sided Lipschitz velocities}\label{sec:osl}

Sections~\ref{sec:convergence}--\ref{subsec:adaptive-energy} are written in the case $v\in L^1((0,T);C^{0,1})$. We indicate here why the restriction is one of exposition, leaving the details to \Appref{sec:appendix-osl}.
Let $\Omega$ be a flat torus, and let $v\in L^1((0,T);L^\infty(\Omega;\R^2))$ be one-sided Lipschitz,
\begin{equation}\label{eq:osl}
\left(v(t,x)-v(t,y)\right)\cdot(x-y)\le L(t)|x-y|^2,
\qquad L\ge0,\quad L\in L^1(0,T),
\end{equation}
for almost every $t$ and all periodic representatives $x,y\in\R^2$. \eqref{eq:osl} bounds expansion but leaves compression free, and it is what keeps the forward flow $\Phi$ of $v$ unique, in the Filippov sense when $v$ jumps.
Transporting the initial measure along $\Phi$ is the notion of solution proposed by Poupaud and Rascle \cite{PoupaudRascle1997} for which we refer to \cite{bouja,LionsSeeger2024,DelarueLagoutiereVauchelet2017}.

Three of our arguments use nothing beyond \eqref{eq:osl}.
Theorem~\ref{thm:W1-convergence} relies on the inequality $\frac{\diff}{\diff t}|x-y|\le L(t)|x-y|+|r|$ between a discrete trajectory $\dot x=v(t,x)+r$ and an exact one, which is \eqref{eq:osl} followed by Gronwall.
Theorem~\ref{thm:alpha-feedback} controls the transport term through \eqref{eq:T_bound_CS}, in which only $\|v(t,\cdot)\|_{L^\infty}$ appears.
The inequality \eqref{eq:adaptive-global-energy-inequality} survives once the integration by parts of \Secref{subsec:adaptive-energy}, which uses $\operatorname{div}v$, is replaced by a summation over the faces of the tessellation, on each of which \eqref{eq:osl} may be applied to the two generators sharing it; this gives $\mathscr T_v\le4L(t)\mathscr F$.

What does not survive is the separation of the generators, which in Theorem~\ref{thm:adaptive-energy-convergence} came from a lower bound on $(v_i-v_j)\cdot(x_i-x_j)$ that \eqref{eq:osl} does not provide.
Collisions do in fact occur. On $(\R/\mathbb Z)^2$,  $v(x,y)=(2x-\operatorname{sign}x,0)$ is one-sided Lipschitz with $L=2$, and moves two generators at $(\pm s,0)$, $0<s<1/4$, according to $\dot s=2s-1+2^{5/4}(1/4-s)^3\le-\frac12+2^{-19/4}$.
The remedy is to replace $v$ by the smoothed  $v^\varepsilon=\eta_\varepsilon*v$, which is Lipschitz for each $\varepsilon>0$ and satisfies \eqref{eq:osl} with the same $L$.
We use $\alpha_h=\mathscr G/h^{5/2}$, which is locally Lipschitz by Remark~\ref{rem:lipschitz}.
Theorem~\ref{thm:adaptive-energy-convergence} then gives the convergence estimate for $v^\varepsilon$, with constants independent of $\varepsilon$, thanks to the one-sided energy bound above.
Averaging \eqref{eq:osl} against $\eta_\varepsilon$ bounds the distance between the flows of $v^\varepsilon$ and of $v$ by $O(\varepsilon^{1/2})$.
With initial data as in \eqref{eq:adaptive-initial-energy} and $\mu(t):=\Phi(t,\cdot)_\#(\rho_0\,\diff x)$, we obtain
\begin{equation}\label{eq:osl-convergence}
\sup_{t\in[0,T]}W_1\left(\mu^{h,\varepsilon}(t),\mu(t)\right)\le C\|\rho_0\|_{L^1}\left(h^{1/4}+\varepsilon^{1/2}\right)
\end{equation}
for $h,\varepsilon\in(0,1]$, so that $\varepsilon(h)=h$ retains the rate $O(h^{1/4})$.

\FloatBarrier
\section{A numerical illustration} \label{sec:num}

We illustrate the effect of the Lloyd correction on a standard test case for the compressible Euler equations
\begin{align}
\partial_t \rho + \nabla \cdot(\rho v) &= 0, \label{eq:mass}\\
\partial_t(\rho v) + \nabla \cdot(\rho v\otimes v + p I_2) &= 0, \label{eq:mom}\\
\partial_t\left(\rho E\right) + \nabla \cdot\left(\left(\rho E + p\right)v\right) &= 0, \label{eq:energy}
\end{align}
where $\rho$ is the density, $v=(v_1,v_2)$ the velocity, $p$ the pressure, $I_2$ the $2\times2$ identity matrix, and $E=e+\frac12|v|^2$ the total specific energy, $e$ being the internal energy.
The first equation is the continuity equation \eqref{eq:continuity}.
The system is closed by an equation of state; for an ideal gas, $p=(\gamma-1)\rho e$ with $\gamma>1$.
We impose the impermeability condition $v\cdot n=0$ on $(0,T)\times\partial\Omega$, and the domain is $\Omega=(0,1)^2$.
There is a substantial existence theory for \eqref{eq:mass}--\eqref{eq:energy}, ranging from local smooth solutions to weak and entropy solutions; see for instance \cite{Majda1984,Dafermos2016}.

The scheme of \cite{despres2024,desprem} is at the same time a particle method, in the sense of \eqref{eq:rho-def}--\eqref{eq:lag-step-0}, and a finite volume method on the Vorono\"i tessellation of the particles, which is recomputed at every time step.
The velocities of the particles are computed with a Riemann solver, for which we refer to those references.

The initial data are those of a Riemann problem, with $\rho_L=1$, $v_L=(0,0)$ and $p_L=1$ for $x<0.5$, and $\rho_R=0.125$, $v_R=(0,0)$ and $p_R=0.1$ for $x>0.5$, independently of $y$.
The solution is a weak solution consisting of a shock, followed by a contact discontinuity and a rarefaction fan.
The velocity jumps across the shock, so this example lies outside the hypotheses of Theorem~\ref{thm:W1-convergence}; the figures below illustrate the effect of the Lloyd correction, but the convergence theorem does not apply to them.
The strong gradients also trigger numerical instabilities, the particles tending to come extremely close to one another; see \cite{despres2024,desprem}.
In \cite{loulou,kincl}, this is mitigated by a regularization of ALE type.
Here we rely instead on Lloyd's algorithm \eqref{eq:lloyd-step}, which acts as an additional repulsive force between the particles.

Figure~\ref{fig:1} shows the Vorono\"i tessellation at the first recorded time $t=0.0052$, at the intermediate time $t\approx 0.08$ and at the final time $t\approx 0.16$, for the computation without Lloyd correction and with the same $\gamma_L=\gamma_R=1.4$ on both sides.
Some generators come so close to each other that the construction of the Vorono\"i tessellation becomes unstable at the final time.

Figure~\ref{fig:2} shows the same computation with the Lloyd correction.
The generators remain separated, and the computation runs until $t\approx 0.16$ without noticeable instability.

Finally, Figure~\ref{fig:3} shows the same computation with the Lloyd correction and two different adiabatic exponents, $\gamma_L=1.4$ and $\gamma_R=\frac 53$.
We plot the pressure, the velocity and the density at all generators as functions of the horizontal coordinate (in green), together with a reference solution computed on a very fine mesh (in magenta).
Since the computation is purely Lagrangian, the two values of $\gamma$ stay on their own side of the contact discontinuity, which ALE methods such as \cite{loulou,kincl} cannot ensure.
The three waves (shock, contact discontinuity and rarefaction) are at the correct positions.
The only noticeable discrepancy is a wall-heating error at the contact discontinuity, visible on the density.
This confirms that the Lloyd correction does not affect the consistency of the method.

\begin{figure}[htbp]
\centering
\includegraphics[width=\linewidth]{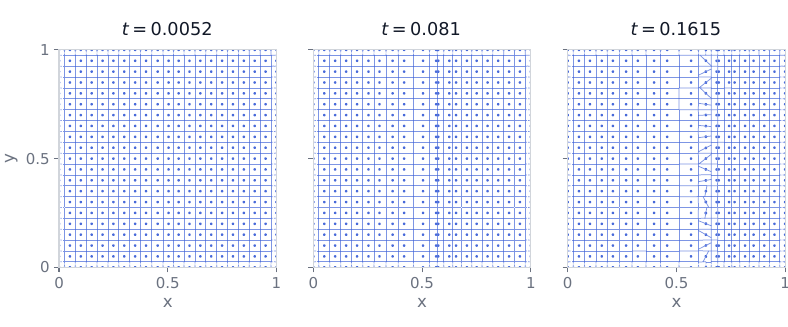}
\caption{Shock calculation on a 2D Vorono\"i mesh, without Lloyd algorithm.
From left to right, the first recorded snapshot $t=0.0052$, then $t=0.081$ and $t=0.1615$.}
\label{fig:1}
\end{figure}

\begin{figure}[htbp]
\centering
\includegraphics[width=\linewidth]{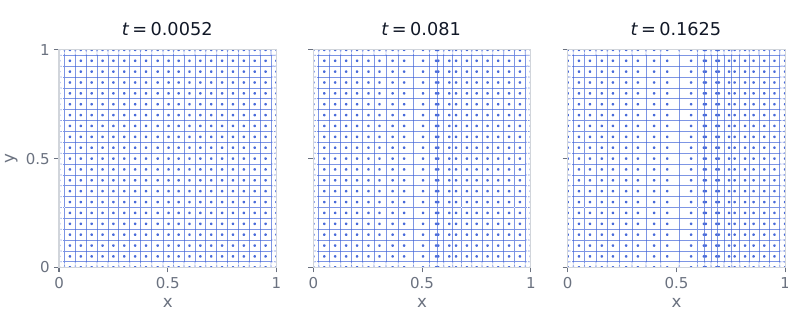}
\caption{Shock calculation on a 2D Vorono\"i mesh, with Lloyd algorithm.
From left to right, the first recorded snapshot $t=0.0052$, then $t=0.081$ and $t=0.1625$.}
\label{fig:2}
\end{figure}

\begin{figure}[htbp]
\centering
\includegraphics[width=\linewidth]{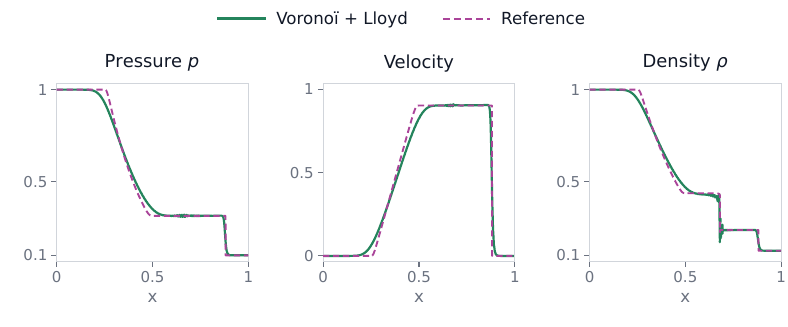}
\caption{Bi-fluid Riemann problem ($\gamma_L=1.4$ and $\gamma_R=\frac 53$).
Numerical solution calculated with a Vorono\"i+Lloyd method versus a reference solution.
From left to right, the pressure, the velocity and the density.}
\label{fig:3}
\end{figure}


\section{Outlook}\label{sec:outlook}

There are several avenues that remain to be explored regarding various enhancements of the proposed stabilization. 

A first one is whether Laguerre cells can provide improved theoretical or numerical conclusions. Therein, the weight carried by each particle is an extra degree of freedom, which ought to let the mesh follow a strongly inhomogeneous density without breaking the semi-discrete transport structure our estimates rely on.

A second avenue concerns the centroid itself, which is only one choice of repulsion among many, and it comes with the energy $\mathscr F$. But energies of this kind appear in a variety of other settings including in the mean-field description of Transformers, where tokens attract or repel through an attention kernel whose geometry decides whether they cluster \cite{geshkovski2023emergence,geshkovski2025mathematical,alcalde2025attention}. 

Finally, investigating which mechanism allows to generalize these results to a BV velocity, and at what cost in consistency, is also of interest.

\smallskip

\noindent
\textbf{Funding.} This work received funding by Agence Nationale de la Recherche, program France 2030, reference ANR-23- PEIA-0004.
BG’s research was supported by a Sorbonne Emergences grant and a gift from Google.

\smallskip

\noindent
\textbf{AI use.} AI-based tools were used for improving the plotting scripts of an existing code for compressible Euler, and for proofreading the manuscript. 
ChatGPT Pro was also used to obtain the extension of all results to the one-sided Lipschitz case presented in \Secref{sec:osl}, as well as for generating the Lean formalization, which were then reviewed by the authors.  The authors have verified all AI-assisted output and take full responsibility for the content of the paper.

\smallskip

\noindent
\textbf{Lean formalization.} Parts of \Appref{sec:appendix-osl} have been formalized in Lean.
The formalization proves $\sum_\Gamma d_\Gamma\int_\Gamma f\,\diff s=4\mathscr F$ and $\mathscr T_v\le4L\mathscr F$, including the pairing of the faces and the cancellation of the faces a cell shares with its own translates.
It also proves that convolution preserves the one-sided constant, the cancellation coming from $\int z\,\eta_\varepsilon(z)\,\diff z=0$, the Gronwall bound on the trajectory error, and the resulting $O(\varepsilon^{1/2})$ bound in Wasserstein distance.
The face representation, the local divergence formulas, the existence of trajectories and the one-sided bound for Filippov selections are left as hypotheses.
The files are at \url{https://github.com/borjang/2026-lloyd-transport}.

\appendix

\section{An envelope theorem for the gradient}\label{sec:appendix-envelope}

Let $\Omega\subset\mathbb R^2$ be a bounded Lipschitz domain and let $N\in\mathbb N$.
For a configuration $X=(x_1,\dots,x_N)\in\Omega^N$ with pairwise distinct generators, let
\[
\phi(X;x):= \min_{1\le i\le N} |x-x_i|^2,
\qquad x\in\Omega .
\]
The Vorono\"i skeleton $\mathscr S(X):=\bigcup_{i\neq j}\{x\in\Omega:\ |x-x_i|=|x-x_j|\}$ is contained in a finite union of lines, hence has measure zero, and for every $x\in\Omega\setminus\mathscr S(X)$ there is a unique index $i(x)$ with $\phi(X;x)=|x-x_{i(x)}|^2$, namely the index with $x\in\Omega_{i(x)}(X)$.
Therefore the energy \eqref{eq:energy-F} satisfies $\mathscr F(X)=\int_\Omega\phi(X;x)\,\diff x$.

The following result is known in the literature, but we do not know whether the proof is new.
We owe this version to Rapha\"el Berthier.

\begin{proposition}\label{prop:env-gradient}
Let $\mathcal U\subset\Omega^N$ be an open set such that the generators of every $X\in\mathcal U$ are pairwise distinct.
Then $\mathscr F\in C^1(\mathcal U)$ and
\begin{equation}\label{eq:env-gradient}
\nabla_{x_i}\mathscr F(X)=2\int_{\Omega_i(X)}(x_i-x)\,\diff x
=2\,V_i(X)\,\left(x_i-c_i(X)\right),
\qquad 1\le i\le N.
\end{equation}
\end{proposition}

\begin{proof}
Fix $X\in\mathcal U$ and a direction $H=(h_1,\dots,h_N)\in(\mathbb R^2)^N$, and set $X^\varepsilon:=X+\varepsilon H$ for $\varepsilon$ small.
For $x\in\Omega$ write $f_i(\varepsilon;x):=|x-x_i-\varepsilon h_i|^2$, so that $\phi(X^\varepsilon;x)=\min_{1\le i\le N}f_i(\varepsilon;x)$.
For $x\in\Omega\setminus\mathscr S(X)$, the minimum at $\varepsilon=0$ is attained at the single index $i(x)$, with a strict inequality against every other index.
By continuity the same index attains the minimum for $|\varepsilon|$ small, and therefore
\[
\frac{\phi(X^\varepsilon;x)-\phi(X;x)}{\varepsilon}
=\frac{f_{i(x)}(\varepsilon;x)-f_{i(x)}(0;x)}{\varepsilon}
\longrightarrow2\,(x_{i(x)}-x)\cdot h_{i(x)}
\qquad(\varepsilon\to0).
\]
On the other hand, $f_i(\varepsilon;x)-f_i(0;x)=-2\varepsilon (x-x_i)\cdot h_i + \varepsilon^2|h_i|^2$, so that for $|\varepsilon|\le 1$
\[
\left|\frac{\phi(X^\varepsilon;x)-\phi(X;x)}{\varepsilon}\right|
\le
\max_{1\le i\le N}\left|\frac{f_i(\varepsilon;x)-f_i(0;x)}{\varepsilon}\right|
\le C_2,
\]
where $C_2:=\max_{1\le i\le N}(2C_1|h_i|+|h_i|^2)$ and $C_1:=\operatorname{diam}\Omega$ bounds $|x-x_i|$.
The difference quotients are thus dominated by the constant $C_2$, which is integrable on the bounded set $\Omega$, and dominated convergence gives
\[
\lim_{\varepsilon\to0}\frac{\mathscr F(X^\varepsilon)-\mathscr F(X)}{\varepsilon}
=
\int_\Omega 2\,(x_{i(x)}-x)\cdot h_{i(x)}\,\diff x
=
2\sum_{i=1}^N\int_{\Omega_i(X)} (x_i-x)\cdot h_i\,\diff x,
\]
since $\{x\in\Omega:\ i(x)=i\}=\Omega_i(X)$ up to a set of measure zero.
This is linear and continuous in $H$, and it depends continuously on $X$ because the cell volumes and centroids do.
Hence $\mathscr F$ is continuously differentiable on $\mathcal U$, with gradient \eqref{eq:env-gradient}.
\end{proof}

The same proof applies on $\Omega=\R^2/\mathcal L$, where $|x-x_i|=\min_{k\in\mathcal L}|\tilde x-\tilde x_i-k|$ for representatives $\tilde x,\tilde x_i\in\R^2$. For $\tilde x$ in a bounded subset of $\R^2$, only finitely many $k\in\mathcal L$ are relevant in this minimum, and the argument is unchanged.

\begin{remark}\label{rem:lipschitz}
The maps $X\mapsto V_i(X)$ and $X\mapsto V_i(X)c_i(X)=\int_{\Omega_i(X)}x\,\diff x$ are locally Lipschitz on $\mathcal U$, whether or not the Vorono\"i neighbors change.
Indeed, each cell $\Omega_i(X)$ is the intersection of $\Omega$ with finitely many half-planes, bounded by the bisectors of $x_i$ and $x_j$, $j\ne i$ (on the torus, of $x_i$ and $x_j+k$, $k\in\mathcal L$).
These lines depend smoothly on $X$ as long as the generators stay separated, so the symmetric difference of $\Omega_i(X)$ and $\Omega_i(X')$ has area $O(|X-X'|)$, locally uniformly on $\mathcal U$.
Since $V_i>0$ on $\mathcal U$, the centroids $c_i$ and the energy $\mathscr G$ defined in \eqref{eq:energy-G} are locally Lipschitz on $\mathcal U$ as well.
\end{remark}

\section{Complements on one-sided Lipschitz velocities}\label{sec:appendix-osl}

We gather here the computations left aside in \Secref{sec:osl}.
Throughout, we abbreviate $V_T:=\int_0^T\|v(t,\cdot)\|_{L^\infty}\,\diff t$.
When $v(t,\cdot)$ is continuous write $v_i:=v(t,x_i)$, and when it is not, let $v_i$ be any element of the Filippov set $K[v](t,x_i)$, the intersection over the neighborhoods of $x_i$ of the closed convex hulls of the essential values of $v(t,\cdot)$ there.
These velocities obey the same one-sided bound $(u-w)\cdot(x-y)\le L(t)|x-y|^2$ for all $u\in K[v](t,x)$ and $w\in K[v](t,y)$, and almost every $t$.
We take essential values of $v(t,\cdot)$ at points near $x$ and near $y$ and pass to the limit in \eqref{eq:osl}, which is justified because $v$ is bounded, and we can then use that the left-hand side is linear in $u$ and in $w$ to reach their convex hulls.

Consider first the term $\mathscr T_v$.
In \Secref{subsec:adaptive-energy} it was estimated by an integration by parts, which requires a derivative of $v$ and is no longer available.
We use instead that \eqref{eq:osl} is an inequality between two points, and that every face of the tessellation is associated with a pair of generators.
Let $X\in\Omega^N$ have pairwise distinct generators, let $\{P_i\}$ be the periodic cells of \Secref{subsec:adaptive-energy}, and set $f(y):=\min_j|y-x_j|^2$, so that $\int_\Omega f=\mathscr F(X)$.
Let $\Gamma$ run over one representative of each unoriented face, including the faces shared by a cell and a translate of itself, and, when $\Gamma$ separates $P_i$ from $P_j+k$, write $d_\Gamma:=|x_i-x_j-k|$ and $n_\Gamma:=(x_j+k-x_i)/d_\Gamma$.
Since $f=|y-x_i|^2=|y-x_j-k|^2$ on $\Gamma$, the identity \eqref{eq:Tv_faces} may be rewritten as
\[
\mathscr T_v=-\sum_\Gamma\int_\Gamma f\left(v_i-v_j\right)\cdot n_\Gamma\,\diff s
=\sum_\Gamma\frac{\left(v_i-v_j\right)\cdot\left(x_i-x_j-k\right)}{d_\Gamma}\int_\Gamma f\,\diff s .
\]
Each numerator is the left-hand side of \eqref{eq:osl} for the pair $x_i$ and $x_j+k$, and is therefore at most $L(t)d_\Gamma^2$.
It remains to identify the weighted sum of face integrals, which follows from the divergence theorem.
Indeed $\operatorname{div}_y\left(|y-x_i|^2(y-x_i)\right)=4|y-x_i|^2$, while $(y-x_i)\cdot n_i=d_\Gamma/2$ on $\Gamma$, since $\Gamma$ lies on the bisector of $x_i$ and $x_j+k$; integration over each $P_i$ and summation over $i$, in which every face is counted once from each side, give $\sum_\Gamma d_\Gamma\int_\Gamma f\,\diff s=4\mathscr F(X)$.
Therefore $\mathscr T_v\le4L(t)\mathscr F(X)$, and \eqref{eq:adaptive-energy-identity} becomes \eqref{eq:adaptive-global-energy-inequality} as before.

The separation of the generators requires another argument.
While the Lloyd term is repulsive, it doesn't prevent collisions by itself. Indeed let $\Omega=(\R/\mathbb Z)^2$ and let $v(x,y):=\left(2x-\operatorname{sign}x,\,0\right)$ for $-1/2<x<1/2$, extended periodically.
Its jump at the origin is compressive, and its one-sided Lipschitz constant is $L=2$.
Let the two generators be at $(\pm s,0)$ with $0<s<1/4$, so that $N=2$ and $h=2^{-1/2}$.
Their cells are the vertical strips separated by $x=0$ and $x=1/2$, of volume $1/2$ and with centroids $(\pm1/4,0)$, so that $\mathscr G=(1/4-s)^2$ and the configuration remains symmetric.
With $\alpha_h=\mathscr G/h^{5/2}$ we obtain
\[
\dot s=2s-1+2^{5/4}\left(1/4-s\right)^3\le-\tfrac12+2^{-19/4}<0 ,
\]
the drift towards the origin being of order one, while the Lloyd correction does not exceed $2^{-19/4}$.
The two generators therefore collide in finite time.
In the proof of Theorem~\ref{thm:adaptive-energy-convergence}, they were kept apart by a lower bound on $\left(v(t,x_i)-v(t,x_j+k)\right)\cdot w$, with $w=x_i-x_j-k$, which \eqref{eq:osl} does not provide.

This difficulty can be avoided by regularizing the velocity in space, at a cost which we now estimate.
Let $\eta_\varepsilon$ be a nonnegative symmetric mollifier of integral one supported in $B(0,\varepsilon)$, and let $v^\varepsilon:=\eta_\varepsilon*v$ denote the periodic convolution.
Averaging \eqref{eq:osl} shows that $v^\varepsilon$ satisfies it with the same $L(t)$, and $v^\varepsilon(t,\cdot)$ is Lipschitz for each fixed $\varepsilon>0$.
For the locally Lipschitz feedback $\alpha_h=\mathscr G/h^{5/2}$, Theorem~\ref{thm:adaptive-energy-convergence} therefore gives a global solution with distinct generators.
With the face estimate above in place of the integration by parts, its proof gives the energy and convergence bounds with constants independent of $\varepsilon$.
The Lipschitz constant of $v^\varepsilon$ is not controlled uniformly in $\varepsilon$, but it is needed only to produce the solution and to keep the generators apart at fixed $\varepsilon$.
The energy and transport inequalities do not need it. They rely on the one-sided constant $L$, which is the same for $v^\varepsilon$ and $v$, and on $\|v^\varepsilon(t,\cdot)\|_{L^\infty}\le\|v(t,\cdot)\|_{L^\infty}$, so that the constants above depend only on $C_1$, $T$, $\Omega$, $\Lambda(T)$ and $V_T$.
It remains to compare the two flows.
Let $x(\cdot)$ and $x^\varepsilon(\cdot)$ be continuous lifts to $\R^2$ of trajectories of $v$ and of $v^\varepsilon$ issued from the same point, and set $e:=x^\varepsilon-x$ and $w:=\dot x\in K[v](t,x)$.
Applying \eqref{eq:osl} to the pair $x^\varepsilon-z$ and $x$, and averaging against $\eta_\varepsilon(z)\,\diff z$, the terms in $z$ carrying $w$ vanish since $\int z\,\eta_\varepsilon(z)\,\diff z=0$, and we obtain
\[
\frac12\frac{\diff}{\diff t}|e|^2=\left(v^\varepsilon(t,x^\varepsilon)-w\right)\cdot e
\le L(t)\left(|e|^2+\varepsilon^2\right)+\varepsilon\|v(t,\cdot)\|_{L^\infty} .
\]
Since $e(0)=0$, Gronwall's lemma yields $\sup_{t\in[0,T]}|e(t)|^2\le 2e^{2\Lambda(T)}\left(\varepsilon^2\Lambda(T)+\varepsilon V_T\right)$, uniformly in the initial point.
Coupling the two flows by their common initial measure, we deduce that $\sup_{t\in[0,T]}W_1(\mu^\varepsilon(t),\mu(t))\le C\|\rho_0\|_{L^1}\varepsilon^{1/2}$ for $\varepsilon\le1$, and \eqref{eq:osl-convergence} follows from \eqref{eq:adaptive-transport-convergence}, applied to $v^\varepsilon$, and the triangle inequality.

\bibliographystyle{plain}
\bibliography{refs-v7}

@book{Majda1984,
  author    = {Andrew Majda},
  title     = {Compressible Fluid Flow and Systems of Conservation Laws in Several Space Variables},
  publisher = {Springer},
  year      = {1984},
  series    = {Applied Mathematical Sciences},
  volume    = {53}
}

@book{Dafermos2016,
  author    = {Constantine M. Dafermos},
  title     = {Hyperbolic Conservation Laws in Continuum Physics},
  publisher = {Springer},
  edition   = {4},
  year      = {2016},
  series    = {Grundlehren der mathematischen Wissenschaften},
  volume    = {325}
}

@book{villani,
  title={Optimal transport: old and new},
  author={Villani, C{\'e}dric},
  volume={338},
  year={2009},
  publisher={Springer}
}

@article{bouja,
  title={Uniqueness and weak stability for multi-dimensional transport equations with one-sided Lipschitz coefficient},
  author={Bouchut, Francois and James, Francois and Mancini, Simona},
  journal={Annali della Scuola Normale Superiore di Pisa},
  volume={4},
  number={1},
  pages={1--25},
  year={2005}
}

@article{despres2024,
  author  = {Bruno Despr\'es},
  title   = {{Compressible Lagrangian Voronoi meshes and particle dynamics with shocks}},
  journal = {Computer Methods in Applied Mechanics and Engineering},
  volume  = {419},
  pages   = {116648},
  year    = {2024}
}

@article{Springel2010,
  author  = {Volker Springel},
  title   = {E pur si muove: {G}alilean-invariant cosmological hydrodynamical simulations on a moving mesh},
  journal = {Monthly Notices of the Royal Astronomical Society},
  volume  = {401},
  number  = {2},
  pages   = {791--851},
  year    = {2010}
}

@article{MerigotMirebeau2016,
  author  = {Quentin M\'erigot and Jean-Marie Mirebeau},
  title   = {Minimal geodesics along volume-preserving maps, through semidiscrete optimal transport},
  journal = {SIAM Journal on Numerical Analysis},
  volume  = {54},
  number  = {6},
  pages   = {3465--3492},
  year    = {2016}
}

@article{GallouetMerigot2018,
  author  = {Thomas Gallou\"et and Quentin M\'erigot},
  title   = {A {L}agrangian scheme \`a la {B}renier for the incompressible {E}uler equations},
  journal = {Foundations of Computational Mathematics},
  volume  = {18},
  number  = {4},
  pages   = {835--865},
  year    = {2018}
}

@article{DuFaberGunzburger1999,
  author  = {Qiang Du and Vance Faber and Max Gunzburger},
  title   = {Centroidal {V}orono\"i tessellations: Applications and algorithms},
  journal = {SIAM Review},
  volume  = {41},
  number  = {4},
  pages   = {637--676},
  year    = {1999},
  doi     = {10.1137/S0036144599352836}
}

@book{GrafLuschgy2000,
  author    = {Siegfried Graf and Harald Luschgy},
  title     = {Foundations of Quantization for Probability Distributions},
  publisher = {Springer},
  year      = {2000},
  series    = {Lecture Notes in Mathematics},
  volume    = {1730},
  doi       = {10.1007/BFb0103945}
}

@article{alcalde2025attention,
  title={Attention's forward pass and Frank-Wolfe},
  author={Alcalde, Albert and Geshkovski, Borjan and Ruiz-Balet, Dom{\`e}nec},
  journal={arXiv preprint arXiv:2508.09628},
  year={2025}
}

@article{geshkovski2025mathematical,
  title={A mathematical perspective on transformers},
  author={Geshkovski, Borjan and Letrouit, Cyril and Polyanskiy, Yury and Rigollet, Philippe},
  journal={Bulletin of the American Mathematical Society},
  volume={62},
  number={3},
  pages={427--479},
  year={2025}
}

@article{Lloyd1982,
  author  = {Stuart P. Lloyd},
  title   = {Least squares quantization in {PCM}},
  journal = {IEEE Transactions on Information Theory},
  volume  = {28},
  number  = {2},
  pages   = {129--137},
  year    = {1982}
}

@inproceedings{MacQueen1967,
  author    = {J. B. MacQueen},
  title     = {Some methods for classification and analysis of multivariate observations},
  booktitle = {Proceedings of the Fifth Berkeley Symposium on Mathematical Statistics and Probability},
  volume    = {1},
  pages     = {281--297},
  year      = {1967}
}

@article{caglioti2018quantization,
  title={Quantization of Measures and Gradient Flows: a Perturbative Approach in the 2-Dimensional Case},
  author={Caglioti, Emanuele and Golse, Fran{\c{c}}ois and Iacobelli, Mikaela},
  journal={Annales de l'Institut Henri Poincar{\'e} C, Analyse non lin{\'e}aire},
  volume={35},
  pages={1531--1555},
  year={2018}
}

@article{DECAMPOS2022114680,
title = {A New Updated Reference Lagrangian Smooth Particle Hydrodynamics algorithm for isothermal elasticity and elasto-plasticity},
journal = {Computer Methods in Applied Mechanics and Engineering},
volume = {392},
pages = {114680},
year = {2022},
author = {Paulo R. Refachinho {de Campos} and Antonio J. Gil and Chun Hean Lee and Matteo Giacomini and Javier Bonet},
}

@article{DEVILLERS1996315,
title = {Queries on Voronoi diagrams of moving points},
journal = {Computational Geometry},
volume = {6},
number = {5},
pages = {315-327},
year = {1996},
note = {Sixth Canadian Conference on Computational Geometry},
author = {O. Devillers and M. Golin and K. Kedem and S. Schirra}
}

@article{DUQUE2023268,
title = {A unified derivation of Voronoi, power, and finite-element Lagrangian computational fluid dynamics},
journal = {European Journal of Mechanics - B/Fluids},
volume = {98},
pages = {268-278},
year = {2023},
author = {Daniel Duque},
}

@article{PhysRevLett.83.1775,
  title = {From Molecular Dynamics to Dissipative Particle Dynamics},
  author = {Flekk\o{}y, Eirik G. and Coveney, Peter V.},
  journal = {Phys. Rev. Lett.},
  volume = {83},
  issue = {9},
  pages = {1775--1778},
  year = {1999},
  month = {Aug},
  publisher = {American Physical Society},
}

@article{GABURRO2020109167,
author = {Elena Gaburro and Walter Boscheri and Simone Chiocchetti and Christian Klingenberg and Volker Springel and Michael Dumbser},
title = {High order direct Arbitrary-Lagrangian-Eulerian schemes on moving Voronoi meshes with topology changes},
journal = {Journal of Computational Physics},
volume = {407},
pages = {109167},
year = {2020},
url = {https://www.sciencedirect.com/science/article/pii/S0021999119308721},
}

@article{kincl,
author = {Kincl, Ondrej and Peshkov, Ilya and Boscheri, Walter},
title = {Semi-implicit Lagrangian Voronoi approximation for the incompressible Navier--Stokes equations},
journal = {International Journal for Numerical Methods in Fluids},
volume = {97},
number = {1},
pages = {88-115},
year = {2025}
}

@article{pasta,
author = {Pasta, J.R.  and Ulam, S.},
title = {Heuristic Numerical Work in Some Problems of Hydrodynamics},
journal = {Mathematical Tables and Other Aids to Computation},
volume = {13},
pages = {1-12},
year = {1959}
}

@article{RUSSO199384,
title = {A Deterministic Vortex Method for the Navier-Stokes Equations},
journal = {Journal of Computational Physics},
volume = {108},
number = {1},
pages = {84-94},
year = {1993},
author = {Giovanni Russo},
}

@MISC{shasha,
       author = {{Shashkov}, Mikhail Y. and {Solovjov}, Andrey V.},
        title = "{Numerical simulation of two-dimensional flows by the free-Lagrangian method}",
         year = 1991,
        month = jan,
        pages = {32767},
       adsurl = {https://ui.adsabs.harvard.edu/abs/1991STIN...9232767S}
}

@article{gaga,
title={Convergence of a Lagrangian discretization for barotropic fluids and porous media flow},
  author={Gallou{\"e}t, Thomas O and Merigot, Quentin and Natale, Andrea},
  journal={SIAM Journal on Mathematical Analysis},
  volume={54},
  number={3},
  pages={2990--3018},
  year={2022},
  publisher={SIAM}
}

@article{caca, 
title={Simulation of multiphase porous media flows with minimising movement and finite volume schemes}, 
volume={30}, 
number={6}, 
journal={European Journal of Applied Mathematics}, 
author={Cances, C.  and Gallou\"et, T. and Laborde, M. and Monsaingeon, L.},
 year={2019}, pages={1123--1152}
 }

@article{desprem,
title = {Remarks on a new particle method},
journal = {Journal of Computational Physics},
volume = {523},
pages = {113662},
year = {2025},
author = {Bruno Després and Ronald Remmerswaal},
}

@article{loulou,
title = {ReALE: A reconnection-based arbitrary-Lagrangian–Eulerian method},
journal = {Journal of Computational Physics},
volume = {229},
number = {12},
pages = {4724-4761},
year = {2010},
author = {Raphaël Loubère and Pierre-Henri Maire and Mikhail Shashkov and Jérôme Breil and Stéphane Galera},
}

@inproceedings{geshkovski2023emergence,
  title={The emergence of clusters in self-attention dynamics},
  author={Geshkovski, Borjan and Letrouit, Cyril and Polyanskiy, Yury and Rigollet, Philippe},
  booktitle={Advances in Neural Information Processing Systems},
  volume={36},
  year={2023}
}

@article{KriegSonnleitner2024,
  author  = {David Krieg and Mathias Sonnleitner},
  title   = {Random points are optimal for the approximation of {Sobolev} functions},
  journal = {IMA Journal of Numerical Analysis},
  volume  = {44},
  number  = {3},
  pages   = {1346--1371},
  year    = {2024},
  doi     = {10.1093/imanum/drad014}
}

@article{ElNmeirLuschgyPages2022,
  author  = {El Nmeir, Rancy and Luschgy, Harald and Pag{\`e}s, Gilles},
  title   = {New approach to greedy vector quantization},
  journal = {Bernoulli},
  volume  = {28},
  number  = {1},
  pages   = {424--452},
  year    = {2022},
  doi     = {10.3150/21-BEJ1350}
}

@article{PoupaudRascle1997,
  author  = {Poupaud, Fran{\c c}ois and Rascle, Michel},
  title   = {Measure solutions to the linear multi-dimensional transport equation with non-smooth coefficients},
  journal = {Communications in Partial Differential Equations},
  volume  = {22},
  number  = {1--2},
  pages   = {337--358},
  year    = {1997},
  doi     = {10.1080/03605309708821265}
}

@article{DelarueLagoutiereVauchelet2017,
  author  = {Delarue, Fran{\c c}ois and Lagouti{\`e}re, Fr{\'e}d{\'e}ric and Vauchelet, Nicolas},
  title   = {Convergence order of upwind type schemes for transport equations with discontinuous coefficients},
  journal = {Journal de Math{\'e}matiques Pures et Appliqu{\'e}es},
  volume  = {108},
  number  = {6},
  pages   = {918--951},
  year    = {2017},
  doi     = {10.1016/j.matpur.2017.05.012}
}

@article{LionsSeeger2024,
  author  = {Lions, Pierre-Louis and Seeger, Benjamin},
  title   = {Transport equations and flows with one-sided {Lipschitz} velocity fields},
  journal = {Archive for Rational Mechanics and Analysis},
  volume  = {248},
  number  = {5},
  pages   = {Paper No.~86},
  year    = {2024},
  doi     = {10.1007/s00205-024-02029-0}
}

@article{GrafLuschgyPages2012,
  author  = {Siegfried Graf and Harald Luschgy and Gilles Pag{\`e}s},
  title   = {The local quantization behavior of absolutely continuous probabilities},
  journal = {The Annals of Probability},
  volume  = {40},
  number  = {4},
  pages   = {1795--1828},
  year    = {2012},
  doi     = {10.1214/11-AOP663}
}

@article{Gruber2004,
  author  = {Peter M. Gruber},
  title   = {Optimum quantization and its applications},
  journal = {Advances in Mathematics},
  volume  = {186},
  number  = {2},
  pages   = {456--497},
  year    = {2004},
  doi     = {10.1016/j.aim.2003.07.017}
}

\noindent\begin{minipage}[t]{.48\textwidth}
    {\footnotesize{\bf Bruno Despr\'es}\par
      Laboratoire Jacques-Louis Lions\par
      Inria \& Sorbonne Université\par
      4 Place Jussieu\par
      75005 Paris, France\par
     \par
      e-mail: \href{mailto:bruno.despres@inria.fr}{\textcolor{blue}{\scriptsize bruno.despres@inria.fr}}
      }
    \end{minipage}\hfill%
    \begin{minipage}[t]{.48\textwidth}
      {\footnotesize{\bf Borjan Geshkovski}\par
      Laboratoire Jacques-Louis Lions\par
      Inria \& Sorbonne Université\par
      4 Place Jussieu\par
      75005 Paris, France\par
     \par
      e-mail: \href{mailto:borjan.geshkovski@inria.fr}{\textcolor{blue}{\scriptsize borjan.geshkovski@inria.fr}}
      }
    \end{minipage}%

\end{document}